\documentclass[11pt]{article}
\usepackage{amsmath,amsfonts,amssymb,amsthm,enumerate,graphicx}
\usepackage{tikz,authblk}
\usepackage{tikz-cd}
\usepackage{subfig}
\usepackage{url}
\usepackage{float}
\usepackage{easyReview}

\newtheorem{theorem}{{\bf Theorem}}[section]
\newtheorem{result}{{\bf Result}}[section]
\newtheorem{claim}{{\bf Claim}}[section]

\newtheorem{example}[theorem]{{\bf Example}}
\newtheorem{lemma}[theorem]{{\bf Lemma}}

\newtheorem{proposition}[theorem]{{\bf Proposition}}

\newtheorem{question}[theorem]{{\bf Question}}

\newcommand{\vg}{\vspace{0.5cm}}

\newcommand{\ol}{\overline}

\newcommand{\ZZ}{ \ensuremath{\mathbb{Z}}}

\begin{document}

\title{Semi-equivelar toroidal maps and their $k$-edge covers}

	\author[1] {Arnab Kundu}
 	\author[2] { Dipendu Maity}
 	\affil[1, 2]{Department of Science and Mathematics,
 		Indian Institute of Information Technology Guwahati, Bongora, Assam-781\,015, India.\linebreak
 		\{arnab.kundu, dipendu\}@iiitg.ac.in/\{kunduarnab1998, dipendumaity\}@gmail.com.}

\date{\today}

\maketitle

\begin{abstract}
If the face\mbox{-}cycles at all the vertices in a map are of same type then the map is called semi\mbox{-}equivelar. In particular, it is called equivelar if the face-cycles contain same type of faces.
A map $K$ on a surface is called edge-transitive if the automorphism group of $K$ acts transitively on the set of edges of $K$.  A tiling is edge-homogeneous if any two edges with vertices of congruent face-cycles. In general, edge-homogeneous maps on a surface form a bigger class than edge-transitive maps. There are edge-homogeneous toroidal maps which are not edge\mbox{-}transitive. 
An edge-homogeneous map is called $k$-edge-homogeneous if it contains $k$ number of edge orbits.  In particular, if $k=1$ then it is called edge-transitive map. In general, a map is called $k$-edge orbital or $k$-orbital if it contains $k$ number of edge orbits. A map is called minimal if the number of edges is minimal. 
A surjective mapping $\eta \colon M \to K$ from a map $M$ to a map $K$ is called a covering if it preserves adjacency and sends vertices, edges, faces of $M$ to vertices, edges, faces of $K$ respectively. Orbani{\' c} et al. and {\v S}ir{\'a}{\v n} et al. have shown that every edge-homogeneous toroidal map has edge-transitive cover. In this article, we show the bounds of edge orbits of edge-homogeneous toroidal maps. Using these bounds, we show the bounds of edge orbits of non-edge-homogeneous semi-equivelar toroidal maps. We also prove that if a edge-homogeneous map is $k$ edge orbital then it has a finite index $m$-edge orbital minimal cover for $m \le k$. We also show the existence and classification of $n$ sheeted covers of edge-homogeneous toroidal maps for each $n \in \mathbb{N}$. We extend this to non-edge-homogeneous semi-equivelar toroidal maps and prove the same results, i.e., if a non-edge-homogeneous map is $k$ edge orbital then it has a finite index $m$-edge orbital minimal cover (non-edge-homogeneous) for $m \le k$ and then classify them for each sheet. 
\end{abstract}

\noindent {\small {\em MSC 2010\,:} 52C20, 52B70, 51M20, 57M60.

 }

\section{Introduction and results}

A {\em map K} is an embedding of a graph {\em G} on a surface {\em S} such that the closure of components of $S \setminus G$, called the {\em faces} of $K$, are homeomorphic to $2$-discs. A map $K$ is said to be a {\em polyhedral map} if the intersection of any two distinct faces is either empty, a common vertex, or a common edge. Here map means polyhedral map.

The $face\mbox{-}cycle$ $C_u$ of a vertex $u$ (also called the {\em vertex-figure} at $u$) in a map is the ordered sequence of faces incident to $u$. So, $C_u$ is of the form $(F_{1,1}\mbox{-}\cdots \mbox{-}F_{1,n_1})\mbox{-}\cdots\mbox{-}(F_{k,1}\mbox{-}$ $\cdots \mbox{-}F_{k,n_k})\mbox{-}F_{1,1}$, where $F_{i,\ell}$ is a $p_i$-gon for $1\leq \ell \leq n_i$, $1\leq i \leq k$, $p_r\neq p_{r+1}$ for $1\leq r\leq k-1$ and $p_n\neq p_1$. The types of the faces in $C_u$ defines the type of $C_u$. In this case, the type of face-cycle($u$) is $[p_1^{n_1}, \dots, p_k^{n_k}]$. A map $M$ is called {\em semi-equivelar} (\cite{DM2018}, we are including the same definition for the sake of completeness) if $C_u$ and $C_v$ are of same type for all $u, v \in V(X)$. More precisely, there exist integers $p_1, \dots, p_k\geq 3$ and $n_1, \dots, n_k\geq 1$, $p_i\neq p_{i+1}$ (addition in the suffix is modulo $k$) such that $C_u$ is of the form as above for all $u\in V(X)$. In such a case, $X$ is called a semi-equivelar map of type (or vertex type) $[p_1^{n_1}, \dots, p_k^{n_k}]$ (or, a map of type $[p_1^{n_1}, \dots, p_k^{n_k}]$). A map $K$ is called {\em equivelar} if $C_u$ and $C_v$ are of same type with same type of faces for all $u, v \in V(K)$. In such a case, $K$ is called a $equivelar$ map.

Two maps of fixed type on the torus are {\em isomorphic} if there exists a {\em homeomorphism} of the torus which sends vertices to vertices, edges to edges, faces to faces and preserves incidents. More precisely, 
if we consider two polyhedral complexes $K_{1}$ and $K_{2}$ then an isomorphism to be a map $f ~:~ K_{1}\rightarrow K_{2}$ such that $f|_{V(K_{1})} : V(K_{1}) \rightarrow V(K_{2})$ is a bijection and $f(\sigma)$ is a cell in $K_{2}$ if and only if $\sigma$ is cell in $K_{1}$. If $K_1 = K_2$, $f$ is called an $automorphism$ and Aut($K_1$) (automorphism group) is the set of all automorphisms of $K_1$. 

An edge in a map has edge-symbol $(m, \ell; u, v)$  if it is incident to vertices of valence $u$ and $v$ and faces of size $m$ and $\ell$. A map is called {\em edge-homogeneous} if all its edges have the same edge symbol $(m, \ell;u, v)$. Note that this requires vertex valences to alternate between $u$ and $v$ around each face, and face sizes to alternate between $m$ and $\ell$ around each vertex. In particular, $u$, $v$, $m$, $\ell$ must satisfy the condition: if $u$ or $v$ is odd, then $m = \ell$, and if $m$ or $\ell$ is odd, then $u = v$.
An edge-transitive map is obviously edge-homogeneous.

\smallskip 

We know from \cite{Orbanic2011, siran2001} the following result.

\begin{proposition} \label{prop1} Let $K$ be an edge-transitive map of edge-symbol $(m, \ell; u, v)$ on the surface $S$. If $S = \mathbb{R}^2$, then $(m, \ell; u, v) = (3, 3; 6, 6), (4, 4; 4, 4),(6, 6; 3, 3), (3, 6; 4, 4)$ or $(4, 4; 3, 6)$. If $S$ is torus, then $(m, \ell; u, v)$ = $(3, 3; 6, 6)$, $(4, 4; 4, 4)$, $(6, 6; 3, 3)$, $(3, 6; 4, 4)$ or $(4, 4; 3, 6)$. 
\end{proposition}

In \cite{JK2020, JK2021},  Karab{\'a}{\v s} et al. are also working in the direction of the classification of edge-transitive maps on the orientable higher genus surfaces. The authors have enumerated and listed the edge-transitive maps for the orientable genus $g = 2, 3, \dots, 28$.

An {\em edge-homogeneous} tiling of the plane $\mathbb{R}^2$ is a tiling of $\mathbb{R}^2$ by regular polygons such that all the vertices of the tiling are of same type. Gr\"{u}nbaum and Shephard \cite{GS1977} showed that there are exactly eleven types of edge-homogeneous tilings on the plane. These types are $[3^6]$, $[4^4]$, $[6^3]$, $[3^4,6^1]$, $[3^3,4^2]$,  $[3^2,4^1,3^1,4^1]$, $[3^1,6^1,3^1,6^1]$, $[3^1,4^1,6^1,4^1]$, $[3^1,12^2]$,  $[4^1,6^1,12^1]$, $[4^1,8^2]$.
Clearly, the maps of types $[3^4,6^1]$, $[3^3,4^2]$,  $[3^2,4^1,3^1,4^1]$,  $[3^1,4^1,6^1,4^1]$, $[3^1,12^2]$,  $[4^1,6^1,12^1]$, $[4^1,8^2]$ are not edge-homogeneous. We know from \cite{DU2005, DM2017, DM2018}  that the edge-homogeneous tilings are unique as edge-homogeneous maps and as a consequence of this we have the following. 

\begin{proposition} \label{propo-1}
All edge-homogeneous maps on the torus are the quotient of an edge-homogeneous tiling on $\mathbb{R}^2$ by a discrete subgroup of the automorphism group of the tiling.
\end{proposition}

The dual map of $K$, denoted by $K^{*}$, lies in the same surface as $K$ and has the property that $(u, e, f)$  is edge  of $K$ if and only if $(f, e^*, u)$ is edge  of $K^*$. Intuitively, the ``dual edges'' $e$ and $e^*$ may be thought of as segments that bisect each other. Hence, we have the following. 

\begin{proposition} (\cite{siran2001}) \label{prop2} $K^*$ is edge-transitive if and only if  $K$ is edge-transitive. 
\end{proposition}

Note that the dual of the maps of type $[4^4]$ are again the maps of type $[4^4]$. The dual of the map of type $[3^6]$, however, is the map of type $[6^3]$. Thus, form Prop. \ref{prop2}, $K$ with edge-symbol $(m, \ell; u, v) = (3, 3; 6, 6)$ is edge-transitive if and only if  $K^*$ with edge-symbol $(m, \ell; u, v) = (6, 6; 3, 3)$ is edge-transitive on the torus. Clearly, Aut$(K)$ = Aut$(K^*)$. Its also easy to observe that if $K = E/H$ for some tiling $E$ and group $H$, then $K^* = E^*/H$. Thus, we study maps of edge-symbols $(m, \ell; u, v)$ = $(3, 3; 6, 6)$, $(4, 4; 4, 4)$, $(3, 6; 4, 4)$ and $(4, 4; 3, 6)$ on the torus.  Hence, we prove the following.

\begin{theorem}\label{thm:no-of-orbits}
Let $K$ be a map of edge-symbol $(m, \ell; u, v)$ on the torus. Let the edges of $K$ form $m$ ${\rm Aut}(X)$ edge orbits. {\rm (a)} If $(m, \ell; u, v) = (3, 3; 6, 6)$, $(6, 6; 3, 3)$ $(3, 6; 4, 4)$ and $(4, 4; 3, 6)$, then $m\leq 3$.
{\rm (b)} If $(m, \ell; u, v) = (4, 4; 4, 4)$, then $m\leq 2$.
{\rm (c)} The bounds in {\rm (a)}, {\rm (b)} are sharp. 
\end{theorem}

We recall some concepts introduced by Coxeter and Moser in \cite{CM1957}. Let $G \leq$Aut($M$) be a discrete group acting on a map $M$ \emph{properly discontinuously} (\cite[Chapter 2]{katok:1992}). This means that each element $g$ of $G$ is associated to a automorphism $h_g$ of $M$ onto itself, in such a way that $h_{gh}$ is always equal to $h_g h_h$ for any two elements $g$ and $h$ of $G$, and $G$-orbit of any edge $e\in E(M)$ is locally finite. Then, there exists $\Gamma \leq $Aut($M$) such that $K = M/\Gamma$. In such a case, $M$ is called a cover of $K$.

A natural question then is: 

\begin{question}\label{ques}
Let $K$ be a edge-homogeneous  map on the torus. Does there exist any map $M (\neq K)$ which is a cover of $K$? Does this cover exist for every sheet, if so, how many? How the edge orbits of $K$ and $M$ are related? 
\end{question}

In this context, we know from \cite{Orbanic2011, siran2001} the following on the torus.

\begin{proposition} \label{datta2020}
If $K$ is a edge-homogeneous toroidal map then there exists a covering $\eta \colon M \to K$ where $M$ is a edge\mbox{-}transitive toroidal map.
\end{proposition}

Here we prove the following. 

\begin{theorem}\label{thm-main1}
{\rm (a)} If $K_1$ is a $m_1$-edge orbital toroidal map of edge-symbol $(m, \ell; u, v) = (3, 3; 6, 6)$, $(3, 6; 4, 4)$ or $(4, 4; 3, 6)$ then there exists a covering $\eta_{k_1} \colon M_{k_1} \to K_1$ where $M_{k_1}$ is $k_1$-orbital for each $k_1 \le m_1$.
{\rm (b)} If $K_2$ is a $m_2$-orbital toroidal map of edge-symbol $(m, \ell; u, v) = (4, 4; 4, 4)$, then there exists a covering $\eta_{k_2} \colon M_{k_2} \to K_2$ where $M_{k_2}$ is $k_2$-orbital for each $k_2 \le m_2$.
\end{theorem}




We extend the above results for non-edge-homogeneous semi-equivelar toroidal maps. 

\begin{theorem} \label{edge-no-of-orbits}
Let $X$ be a non-edge-homogeneous semi-equivelar map on the torus. Let the edges of $X$ form $m$ ${\rm Aut}(X)$-orbits. 
{\rm (a)} If the type of $X$ is  $[3^1, 12^2]$,$[3^2, 4^1, 3^1, 4^1]$ or $[3^1, 4^1, 6^1, 4^1]$  then $m \leq 6$.
{\rm (b)} If the type of $X$ is  $[4^1, 8^2]$  then $m \leq 4$.
{\rm (c)} If the type of $X$ is  $[3^4, 6^1]$  then $m \leq 8$.
{\rm (d)} If the type of $X$ is  $[3^3, 4^2]$  then $m = 3$.
{\rm (e)} If the type of $X$ is  $[4^1,6^1, 12^1]$  then $m \leq 12$.
These bounds are also sharp. 
\end{theorem}

\begin{theorem}\label{edge-thm-main1}
    {\rm (a)} Let $X$ is a $m$-edge orbital toroidal map of type $[3^1, 12^2]$ or $[3^1,4^1,6^1,4^1]$. Then, there exists a covering $\eta_{k} \colon Y_{k} \to X$ where $Y_{k}$ is $k$-edge orbital for each $k \le m$ such that $2$ divides $k$.\\
    {\rm (b)} If $X_5$ is a $6$-edge orbital toroidal map of type $[3^2,4^1,3^1,4^1]$ then there exists a covering $\eta_{k_5} \colon Y_{k_5} \to X_5$ where $Y_{k_5}$ is $3$-edge orbital.\\
    {\rm (c)} If $X_{10}$ is a $8$-edge orbital toroidal map of type $[3^4,6^1]$ then there exists a covering $\eta_{k_{10}} \colon Y_{k_{10}} \to X_{10}$ where $Y_{k_{10}}$ is $2$-edge orbital.\\
    {\rm (d)} If $X_9$ is a $m_9$-edge orbital toroidal map of type $[4^1,6^1, 12^1]$ then there exists a covering $\eta_{k_9} \colon Y_{k_9} \to X_9$ where $Y_{k_9}$ is $k_9$-edge orbital for $(k_9,m_9)=(6,12),(3,6)$.\\
    {\rm (e)} If $X_6$ is a $m_6$-edge orbital toroidal map of type $[4^1,8^2]$ then there exists a covering $\eta_{k_6} \colon Y_{k_6} \to X_6$ where $Y_{k_6}$ is $k_6$-edge orbital for $(k_6,m_6)=(3,4),(2,3)$.
\end{theorem}

\begin{theorem}\label{thm-main2}
Let $X$ be a semi\mbox{-}equivelar toroidal map or $X$ be of type $(m, \ell; u, v) = (4, 4; 3, 6)$. Let $X$ be $k$-edge orbital. Then, there exists a $n$ sheeted covering $\eta \colon Y \to X$ for each $n \in \mathbb{N}$ where $Y$ is $m$-edge orbital for some $m\le k$.
\end{theorem}

\begin{theorem}\label{thm-main3}
Let $X$ be a $n$ sheeted semi\mbox{-}equivelar or of  type $(m, \ell; u, v) = (4, 4; 3, 6)$ $k$-edge orbital toroidal map and $\sigma(n) = \sum_{d|n}d$. Then, there exists different $n$ sheeted $m$-edge orbital covering $\eta_{\ell} \colon Y_{\ell} \to X$ for $ \ell \in \{1, 2, \dots, \sigma(n)\}$, i.e., $Y_{1},$ $Y_{2},$ $\dots,$ $Y_{\sigma(n)}$ are $n$ sheeted $m$-edge orbital covers of $X$ and different upto isomorphism for some $m\le k$.
\end{theorem}

\begin{theorem}\label{thm-main4}
Let $X$ be a $m$-edge orbital semi\mbox{-}equivelar toroidal map $($or of type $(m, \ell; u, v)$ $= (4, 4; 3, 6))$ and $Y$ be a $k$-edge orbital covers of $X$ $($or of type $(m, \ell; u, v) = (4, 4; 3, 6))$. Then, there exists a $k$-edge orbital covering map $\eta \colon Z \to X$ such that $Z$ is minimal.
\end{theorem}

\section{Edge-homogeneous and non-edge-homogeneous tilings of $\mathbb{R}^2$}

\begin{example}
We first present five edge-homogeneous tilings on the plane.
\end{example}

 \vspace{-0.5cm}
\begin{figure}[H]
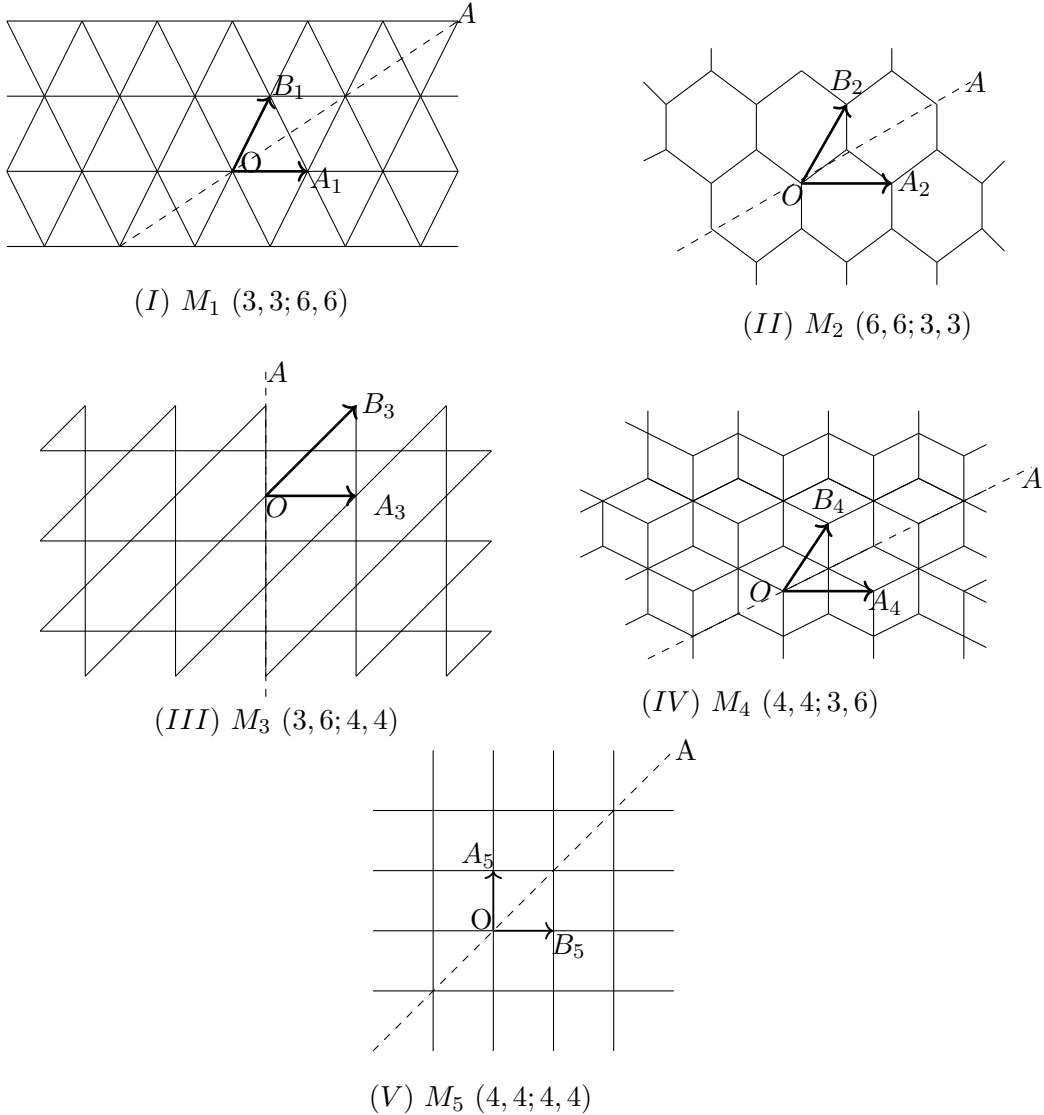

\centering
\begin{minipage}[b][5cm][s]{.45\textwidth}
\centering
\vfill

\vfill
\end{minipage}
\caption{Edge-homogeneous tilings on the plane}\label{fig:eh}
\end{figure}

\begin{example}\label{exm2}
Here we present non edge-homogeneous tilings of the plane.
\end{example}

\begin{figure}[ht]
\centering
\begin{minipage}[b][5cm][s]{.45\textwidth}
\centering
\vfill
%
\vfill
\caption{$E_9$~($[3^1,4^1,6^1,4^1]$)}\label{fig-E_9}
\vspace{\baselineskip}
\end{minipage}\qquad
\vg\vspace{.5cm}
\begin{minipage}[b][5cm][s]{.45\textwidth}
\centering
\vfill
%
\caption{$E_8$~($[3^1,12^2]$)}\label{fig-E_8}
\vfill
\end{minipage}
\end{figure}
\vg
\begin{figure}[ht]
\centering
\begin{minipage}[b][5cm][s]{.45\textwidth}
\centering
\vfill
%
\vfill
\vspace{-1.8cm}
\caption{$E_5$~($[3^2,4^1,3^1,4^1]$)}\label{fig-E_5}
\vspace{\baselineskip}
\end{minipage}\qquad
\vspace{1cm}\vg
\begin{minipage}[b][5cm][s]{.45\textwidth}
\centering
\vfill
%
\caption{$E_4$~($[3^3,4^2]$)}\label{fig-E_4}
\vfill
\end{minipage}

\end{figure}
\vg
\begin{figure}[H]
\centering
\begin{minipage}[b][5cm][s]{.45\textwidth}
\centering
\vfill
%
\vfill
\caption{$E_{10}$~($[3^4,6^1]$)}\label{fig-E_10}
\vspace{\baselineskip}
\end{minipage}\qquad
\vspace{1.5cm}
\begin{minipage}[b][5cm][s]{.44\textwidth}
\centering
\vfill
%
\caption{$E_{11}$~($[4^1,6^1,12^1]$)}\label{fig-E_11}
\vfill
\end{minipage}
\end{figure}
\begin{figure}[H]
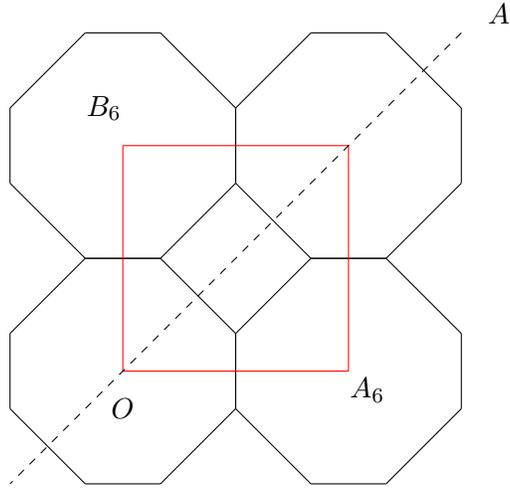

\centering
\begin{minipage}[b][5cm][s]{.45\textwidth}
\centering
\vfill
%
\vfill
\caption{$E_6$~($[4^1,8^2]$)}\label{fig-E_6}
\vspace{\baselineskip}
\end{minipage}\qquad

\end{figure}
\vg

\vspace{2cm}

\begin{example}\label{example_bdd}
We present five edge-homogeneous and non edge-transitive toroidal map. These examples need in the proof of Theorem \ref{thm:no-of-orbits}. These examples attain the upper bounds of the edge orbits for edge-homogeneous maps. 
\end{example}

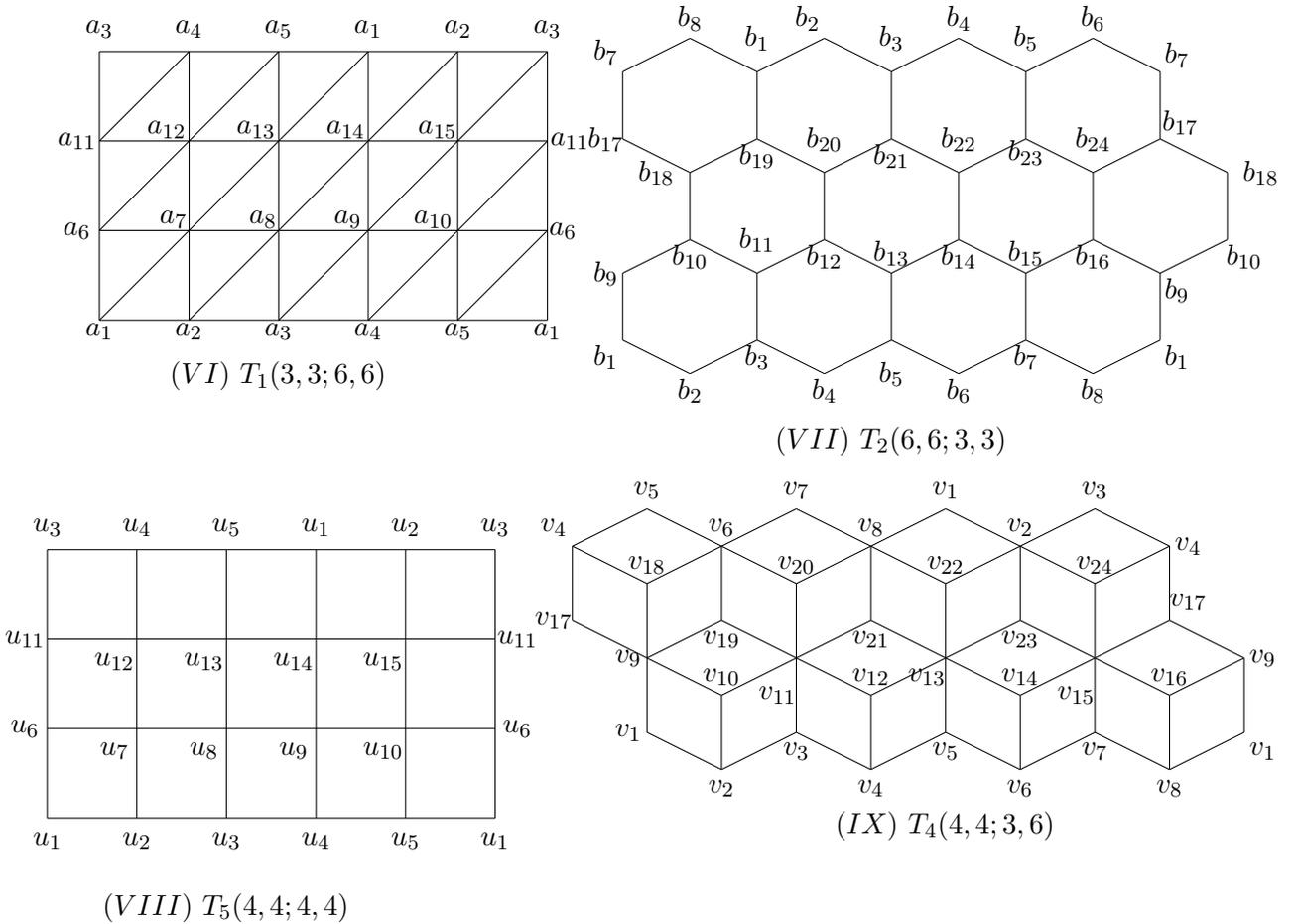
\begin{figure}[H]
\centering
\begin{minipage}[b][4.8cm][s]{.45\textwidth}
\centering
%
\vfill
\end{minipage}
\vg\vg\vg\vg\vg\vg\vg\vg
\caption{Edge-homogeneous maps on the torus}\label{edge:homo}
\end{figure}

\section{Toroidal maps and classification of their edge covers}
Let $\alpha_i, \beta_i$ be two fundamental translations of $M_i$ (in Fig. \ref{fig:eh}) defined by $\alpha_i : z \mapsto z + A_i$ and $\beta_i : z \mapsto z+ B_i$ where $A_i$ and $B_i$ are indicated in i-th figure of Figure \ref{fig:eh}. Let $\rho_i$ be the function on $M_i$ obtained by taking reflection of $M_i$ with respect to the line passing through the origin $O$ and $A$ for $1\le i \le 5$. Let $\chi_i$ be the map obtained by taking reflection of $M_i$ about origin. Let $G_i = \langle \alpha_i,\beta_i,\chi_i \rangle$. 
Let $H_i$ be the translation group generated by $\alpha_i$ and $\beta_i$. So $H_i \le {\rm Aut}(M_i)$. Let for a given map $X$, $E(X)$ denotes the set of all edges of $X$.

\begin{proof}[Proof of Theorem \ref{thm:no-of-orbits}]
Let $X$ be an edge-homogeneous toroidal map with edge symbol $[3,3;$ $6,6]$. Then, by Proposition \ref{propo-1} we can assume that $X=M_1/K_1$ for some subgroup $K_1$ of Aut($M_1$), and $K_1$ has no fixed points. Thus $K_1$ contains only translations and glide reflections. Since $X$ is orientable $K_1$ does not contain any glide reflections.   Let $\eta:M_1\to X$  be the polyhedral map.
Therefore $K_1\le H_1$. Under action of $G_1$, $E(M_1)$ has $3$ orbits.  They are represented by $3$ edges of any triangle in $M_1$. Since, $K_1 \trianglelefteq G_1$ so $G_1/K_1$ is a group. Therefore $G_1/K_1$ acts on $E(X)$. If $O_1, O_2, O_3$ are the $G_1$ orbits of $E(M_1)$ it follows that $\eta(O_1), \eta(O_2), \eta(O_3)$ are the $G_1/K_1$ orbits of $E(X)$. Since, $G_1/K_1 \le {\rm Aut}(X)$ it follows that number of Aut($X$) orbits of $E(X)$ is less than or equal to $3$. In the same way one can see that number of Aut($X$) orbits the maps with edge symbol $(6,6;3,3)$ or $(4,4; 3,6)$ is also less than or equal to $3$.\\
Now suppose $X$ is an edge homogeneous toroidal map with edge symbol $(3,6; 4,4)$. By same argument as before we have  $X=M_3/K_3$ and $K_3 \le H_3$. Observe that under the action of $H_3$ $E(M_3)$ forms $6$ orbits.
Now $E(M_3)$ forms $3$ $G_3\mbox{-}$orbits. Now $K_3 \trianglelefteq G_3$. Thus $G_3/K_3$ is a group. Its action on $E(X)$ also gives $3$ orbits. $G_3/K_3 \le {\rm Aut}(X)$. Therefore $E(X)$ has at most $3$ Aut($X$)\mbox{-}orbits.
This proves part (a) of Theorem \ref{thm:no-of-orbits}.

Now suppose $X$ be a edge-homogeneous toroidal map with edge symbol $[4,4;4,4]$. Proceeding in the similar way as before one can see that under action of $G_2$, $E(X)$ has $2$ orbits. Hence number of Aut($X$) orbits of $E(X)$ is less than or equals to $2$. This proves part (b) of Theorem \ref{thm:no-of-orbits}.

Now we show that there existence of a $3$-orbital edge-homogeneous toroidal map with edge type $(3,3;6,6)$.
Let $X$ be a edge-homogeneous map of type $(3,3;6,6)$. Then $X=M_1/K$ for some discrete fixed point free subgroup $K$ of Aut($M_1$). Aut($X$)$=$ Nor($K$)/$K$. Now $E(X)$ has $3$ $G_1/K$-orbits. If we can show that there exists some $K \le H_1$ such that Nor($K$)$=G_1$ then we are done.\\
Consider $K=\langle \alpha_1^5,\beta_1^3 \rangle$. $\alpha_1^5$ and $\beta_1^3$ are translations by the vectors $5A_1$ and $3B_1$ respectively.\\
${\rm Nor}(K)=\{\gamma \in {\rm Aut}(M_1) \mid \gamma\alpha_1^5\gamma^{-1},\gamma\beta_1^3\gamma^{-1}\in K \} = \{\gamma \in {\rm Aut}(M_1) \mid \gamma(5A_1), \gamma(3B_1) \in \ZZ5A_1+\ZZ3B_1 \}$. Clearly $G_1 \le {\rm Nor}(K)$. But $60$ and $120$ degree rotations and reflection about a line does not belongs to ${\rm Nor}(K)$. Hence ${\rm Nor}(K) = G_1$. With this same method one can see that other bound are sharp also. This proves part (c) of Theorem \ref{thm:no-of-orbits}.

Similarly, by Claims \ref{t-1orb}, \ref{t-5orb}, combinatorially the bounds obtained in (a) and (b) are also sharp.
\end{proof}

\begin{claim}\label{t-1orb}
$T_i$ has exactly three {\rm Aut(}$T_i${\rm )} orbits of edges for $i=1,2,3,4$.
\end{claim}
\begin{proof}
Consider the edge-homogeneous toroidal map $T_1$ as in Figure \ref{edge:homo} (VI). 
Observe that the edges $a_7\mbox{-}a_8, a_7\mbox{-}a_{13}, a_7\mbox{-}a_{12}$ lies on 
cycles with same property but different lengths. Since an automorphism preserves the lengths of cycles so the edges $a_7\mbox{-}a_8, a_7\mbox{-}a_{13}, a_7\mbox{-}a_{12}$ lies in different Aut($T_1$) orbits. The cycles are $$(a_7\mbox{-}a_8\mbox{-}a_9\mbox{-}a_{10}\mbox{-}a_6\mbox{-}a_7), (a_7\mbox{-}a_{13}\mbox{-}a_1\mbox{-}a_7),~ and$$ $$(a_1\mbox{-}a_{12}\mbox{-}a_4\mbox{-}a_9\mbox{-}a_{14}\mbox{-}a_1\mbox{-}a_6\mbox{-}a_{11}\mbox{-}a_3\mbox{-}a_8\mbox{-}a_{13}\mbox{-}a_5\mbox{-}a_{10}\mbox{-}a_{15}\mbox{-}a_2\mbox{-}a_7).$$ 
For the map $T_2$ as in Figure \ref{edge:homo} (VII). 
We consider the edges $b_5\mbox{-}b_{13}, b_{13}\mbox{-}b_{14}, b_{12}\mbox{-}b_{13}$ and corresponding cycles are $$(b_5\mbox{-}b_{13}\mbox{-}b_{14}\mbox{-}b_{22}\mbox{-}b_{23}\mbox{-}b_{5}), (b_{13}\mbox{-}b_{14}\mbox{-}b_{15}\mbox{-}b_{16}\mbox{-}b_{9}\mbox{-}b_{10}\mbox{-}b_{11}\mbox{-}b_{12}\mbox{-}b_{13}),~ and$$ $$(b_{12}\mbox{-}b_{20}\mbox{-}b_{19}\mbox{-}b_1\mbox{-}b_9\mbox{-}b_{16}\mbox{-}b_{24}\mbox{-}b_{23}\mbox{-}b_{5}\mbox{-}b_{13}\mbox{-}b_{12}).$$ 
For the map $T_3$ as in Figure \ref{edge:homo} (X). 
We consider the edges 
$w_3\mbox{-}w_{10}, w_3\mbox{-}w_{4}, w_4\mbox{-}w_{11}$ and the corresponding cycles are $$(w_3\mbox{-}w_{10}\mbox{-}w_{15}\mbox{-}w_{22}\mbox{-}w_{27}\mbox{-}w_{24}\mbox{-}w_3), (w_1\mbox{-}w_{2}\mbox{-}w_{3}\mbox{-}w_{4}\mbox{-}w_{5}\mbox{-}\\w_{6}\mbox{-}w_7\mbox{-}w_8\mbox{-}w_1),~and~(w_4\mbox{-}w_{11}\mbox{-}w_{18}\mbox{-}w_{24}\mbox{-}$$ $$w_{32}\mbox{-}w_{33}\mbox{-}w_{2}\mbox{-}w_{10}\mbox{-}w_{16}\mbox{-}w_{23}\mbox{-}w_{30}\mbox{-}w_{36}\mbox{-}w_{8}\mbox{-}w_{9}\mbox{-}w_{14}\mbox{-}w_{22}\mbox{-}w_{28}\mbox{-}w_{35}\mbox{-}w_{6}\mbox{-}w_{12}\mbox{-}w_{20}\mbox{-}w_{21}\mbox{-}w_{26}\mbox{-}w_{24}\mbox{-}w_{4}).$$ 
For the map $T_4$ as in Figure \ref{edge:homo} (IX). 
We consider the edges $v_4\mbox{-}v_{12}, v_{11}\mbox{-}v_{12}, v_{12}\mbox{-}v_{13}$ and the corresponding cycles are $$(v_4\mbox{-}v_{12}\mbox{-}v_{11}\mbox{-}v_{20}\mbox{-}v_6\mbox{-}v_{14}\mbox{-}v_{13}\mbox{-}v_{22}\mbox{-}v_8\mbox{-}v_{16}\mbox{-}v_{15}\mbox{-}v_{24}\mbox{-}v_2\mbox{-}v_{10}\mbox{-}v_9\mbox{-}v_{18}\mbox{-}v_4),$$ $$ (v_{11}\mbox{-}v_{12}\mbox{-}v_{13}\mbox{-}v_{14}\mbox{-}v_{15}\mbox{-}v_{16}\mbox{-}v_9\mbox{-}v_{10}\mbox{-}v_{11}),~ and~(v_{12}\mbox{-}v_{13}\mbox{-}v_{22}\mbox{-}v_2\mbox{-}v_{10}\mbox{-}v_{11}\mbox{-}v_{20}\mbox{-}v_8\mbox{-}v_{16}\mbox{-}v_9\mbox{-}v_{17}\mbox{-}v_4\mbox{-}v_{12}).$$ 
\end{proof}
\begin{claim}\label{t-5orb}
$T_5$ has exactly two {\rm Aut(}$T_5${\rm )} orbits of edges.
\end{claim}
\begin{proof}
Consider the edge-homogeneous toroidal map $T_5$ as in Figure \ref{edge:homo} (VIII). 
We consider the edges $u_1\mbox{-}u_2, u_1\mbox{-}u_6$ lies on cycles of different length and same type. The cycles are $$(u_1\mbox{-}u_2\mbox{-}u_3\mbox{-}u_4\mbox{-}u_5\mbox{-}u_1), (u_1\mbox{-}u_6\mbox{-}u_{11}\mbox{-}u_3\mbox{-}u_8\mbox{-}u_{13}\mbox{-}u_5\mbox{-}u_{10}\mbox{-}u_{15}\mbox{-}u_2\mbox{-}u_7\mbox{-}u_{12}\mbox{-}u_4\mbox{-}u_9\mbox{-}u_{14}\mbox{-}u_1).$$ The proof follows by the same argument as in Claim \ref{t-1orb}.
\end{proof}

Now we are going to answer Question \ref{ques}. For that we prove Theorem \ref{thm-main1} first.

\begin{proof}[Proof of Theorem \ref{thm-main1}]
Let $X$ be a $3\mbox{-}$orbital edge-homogeneous toroidal map with edge symbol $(3,3;6,6)$. We have to show that there exists a $2\mbox{-}$orbital cover of $X$. By Proposition \ref{prop1} we can assume $X = M_1/K_1$ for some subgroup $K_1$ of Aut($M_1$) such that $K_1$ has no fixed points( vertex, edge or face). Thus $K_1$ consist of only translations and glide reflections. Since $X$ is orientable $K_1$ cannot contain any glide reflection. Thus $K_1 \le H_1$. Let $G_1 := \langle \alpha_1,\beta_1,\chi_1 \rangle$ where $\chi_1$ is the map obtained by taking reflection of $M_1$ about origin.
Now observe that $E(M_1)$ has $3~ G_1\mbox{-}$orbits. Clearly $K_1 \trianglelefteq G_1$. $G_1/K_1$ acts on $E(X)$ and gives $3$ orbits. $G_1/K_1 \le {\rm Aut}(X)$.
Action of Aut($X$) on $E(X)$ also gives $3\mbox{-}$orbits. Therefore the elements $\alpha \in {\rm Aut}(X) \setminus G_1/K_1$ does not change the $G_1/K_1\mbox{-}$orbits of $E(X)$.
Now consider $$ G_1' = \langle \alpha_1 , \beta_1,\chi_1, \rho_1 \rangle  \le {\rm Aut}(X)$$
Observe that under the action of $G_1'$, $E(M_1)$ forms $2$ orbits.
Let $K_1 = \langle \gamma, \delta \rangle \le H_1$ where $\gamma = \alpha^a\circ \beta^b$ and $\delta = \alpha^c\circ \beta^d$ for some $a,b,c,d \in \ZZ$.
Let $\alpha : z \mapsto z+A_1$ and $\beta : z \mapsto z+B_1$. Then $\gamma : z \mapsto z+ C$ and $\delta: z \mapsto z+ D$ where $C = aA_1 + bB_1$ and $D = cA_1+dB_1$.
Now any cover of $X$ will be of the form $M_1/L_1$ for some subgroup $L_1$ of $K_1$ with rank($L_1$)$ = 2.$ Now we made the following claims to get a suitable $L_1$. 
\begin{claim}\label{clm1}
There exists $m \in \ZZ$ such that $L_1 := \langle \gamma^m, \delta^m \rangle \trianglelefteq G_1'.$
\end{claim}
Suppose $L_1 = \langle \gamma^{m_1}, \delta^{m_2} \rangle$. We show that there exists suitable $m_1, m_2$ such that $L_1 \trianglelefteq G_1'$. It turns out that we can take $m_1 = m_2$.
To satisfy $L_1 \trianglelefteq G_1'$ it is enough to show that $\rho_1\gamma^{m_1}\rho_1^{-1}, \rho_1\delta^{m_2}\rho_1^{-1} \in L_1$. It is known that conjugation of a translation by rotation or reflection is also a translation by rotated or reflected vector. Since $\gamma$ and $\delta$ are translation by vectors $C$ and $D$ respectively so $\gamma^{m_1}$ and $\delta^{m_2}$ are translation by vectors $m_1C$ and $m_2D$ respectively. Hence $\rho_1\gamma^{m_1}\rho_1^{-1}$ and $\rho_1\delta^{m_2}\rho_1^{-1}$ are translation by the vectors $C'$ and $D'$ respectively. Where $C' = \rho_1(m_1C) = \rho_1(m_1(aA_1 + bB_1)) = m_1a\rho_1(A_1) + m_1b\rho(B_1) = m_1aB_1 + m_1bA_1$ and similarly $D' = \rho_1(m_2D) = m_2cB_1 + m_2dA_1$.
Now these translations belong to $L_1$ if the vectors $C'$ and $D'$ belong to lattice of $L_1 = \ZZ(m_1C) + \ZZ(m_2D)$.
Let $C', D' \in \ZZ(m_1C) + \ZZ(m_2D)$. Then $\exists~ p, q, s, t \in \ZZ$ such that 
$$C' = p(m_1C) + q(m_2D), ~ D' = s(m_1C) + t(m_2D).$$
Putting expressions of $C', D', C, D$ in above equations we get,
$$(bm_1 - pam_1 - qcm_2)A_1 + (am_1 - pbm_1 - qdm_2)B_1 = 0$$
$$(cm_2 - sbm_1 - tdm_2)A_1 + (dm_2 - sam_1 - tcm_2)B_1 = 0.$$
Since, $\{A_1, B_1\}$ is a linearly independent set we have,
$$pam_1 + qcm_2 = bm_1, ~ pbm_1 + qdm_2 = am_1, ~sbm_1 + tdm_2 = cm_2, ~ sam_1 + tcm_2 = dm_2.$$
Now as rank($L_1$)$=2$ so $m_1, m_2 \neq 0$. Dividing the above system by $m_1m_2$ we get,
$$\frac{pa}{m_2} + \frac{qc}{m_1} = \frac{b}{m_2}, ~\frac{pb}{m_2} + \frac{qd}{m_1} = \frac{a}{m_2}, ~\frac{sb}{m_2} + \frac{td}{m_1} = \frac{c}{m_1}, ~\frac{sa}{m_2} + \frac{tc}{m_1} = \frac{d}{m_1}. $$
Now consider $p, q, s, t$ as variables. We can treat above system as a system of linear equations.  We can write this system in matrix form as follows.
$$\begin{bmatrix} a/m_2 & c/m_1 &0&0\\ 
b/m_2 & d/m_1 &0 &0 \\
0& 0& a/m_2 &c/m_1 \\
0&0& b/m_2&d/m_1\end{bmatrix} 
\begin{bmatrix}
p\\q\\s\\t
\end{bmatrix} =
\begin{bmatrix}
b/m_2\\a/m_2\\d/m_1\\c/m_1
\end{bmatrix}$$
Now $C,D$ are linearly independent thus $ad-bc \neq 0$. Hence the coefficient matrix of the above system has non zero determinant. Therefore the system has an unique solution. After solving we get,
$$p = \frac{m_1^2m_2(d-a)b}{(ad-bc)^2}
,~~q = \frac{m_1^2m_2(a^2-bc)}{(ad-bc)^2},~~s = \frac{m_1m_2^2(d^2-bc)}{(ad-bc)^2}
,~~t = \frac{m_1m_2^2(a-d)c}{(ad-bc)^2}.$$ 
Now if we take $m_1 = m_2  = |ad-bc| = m$(say) then $p, q, s, t \in \ZZ$.
Let $L_1 := \langle \gamma^m, \delta^m \rangle$. Then we have $L_1 \trianglelefteq G_1'.$ Hence our Claim \ref{clm1} proved.
\begin{claim}\label{normal}
$K_1/L_1 \trianglelefteq {\rm Aut}(M_1/L_1)$.
\end{claim}
Let $\rho \in {\rm Nor_{Aut(M_1)}}(L_1)$. Then $\rho \gamma^m \rho^{-1}, \rho \delta^m \rho^{-1} \in L_1$. $\rho \gamma^m \rho^{-1}$ is translation by the vector $(\rho \gamma^m \rho^{-1})(0)$. That is 
$(\rho \gamma^m \rho^{-1})(0) \in$ lattice of $L_1$. Thus there exists $n_1, n_2$ such that $(\rho \gamma^m \rho^{-1})(0) = n_1\gamma^m(0) + n_2 \delta^m(0)$. As $K_1$ is generated by $\gamma$ and $\delta$ so $\rho \gamma \rho^{-1}(0) = n_1\gamma(0) + n_2 \delta(0)$. Thus $\rho \gamma \rho^{-1} \in K_1.$ similarly, $\rho \delta \rho^{-1} \in K_1.$ Therefore $\rho \in {\rm Nor}(K_1) \implies {\rm Nor_{Aut(M_1)}}(L_1) \le {\rm Nor_{Aut(M_1)}}(K_1) \implies K_1 \trianglelefteq {\rm Nor_{Aut(M_1)}}(L_1) \implies K_1/L_1 \trianglelefteq {\rm Nor_{Aut(M_1)}}(L_1)/L_1 $. From \cite{drach:2019} we know ${\rm Aut}(M_1/L_1) = {\rm Nor_{Aut(M_1)}}(L_1)/L_1 .$ This proves  Claim \ref{normal}.

Now by Claim \ref{clm1} $G_1'/L_1$ is a group and acts on $E(M_1/L_1).$ Clearly $E(M_1/L_1)$ has $2~~ G_1'/L_1\mbox{-}$orbits. Since $L_1$ contains two independent vectors, it follows that $Y_1 := M_1/L_1 $ is a toroidal map and $v+L_1 \mapsto v+K_1$ is a covering $\eta : Y_1 \to X$.
Our next aim is to show that $Y$ is a $2$-orbital map. For that we need the following,
\begin{result}\label{quo}
	Let $L \trianglelefteq K$ and $K$ acts on a topological space $E$. Then $\dfrac{E/L}{K/L}$ is homeomorphic to $E/K$ and $\phi:\dfrac{E/L}{K/L} \to E/K$ defined by $\Big(\dfrac{K}{L}\Big)(Lv) \mapsto Kv ~\forall~ v \in E$ is a homeomorphism.
\end{result}
Let $p:M_1/L_1 \to \dfrac{M_1/L_1}{K_1/L_1} = M_1/K_1$ be the quotient map.

\begin{claim}\label{diag}
Given $\alpha \in {\rm Aut}(M_1/L_1)={\rm Nor_{Aut(M_1)}}(L_1)/L_1$ there exists $\widetilde{\alpha} \in {\rm Aut}(M_1/K_1)$ such that $p\circ \alpha = \widetilde{\alpha} \circ p.$
\end{claim}

By Result \ref{quo} we can think $M_1/K_1$ as $\dfrac{M_1/L_1}{K_1/L_1}.$ We show that $\alpha$ takes orbits to orbits for the action of $K_1/L_1$ on $M_1/L_1$.
Let $\mathcal{O}(\ol{v})$ denotes $K_1/L_1$-orbit of $\ol{v} \in M_1/L_1.$ Then
\begin{equation*}
    \begin{split}
        \alpha(\mathcal{O}(\ol{v})) &= \alpha\left(\frac{K_1}{L_1}(\ol{v})\right)  \\&=
        \frac{K_1}{L_1}(\alpha(\ol{v}))~[{\rm since }~ \alpha \frac{K_1}{L_1} = \frac{K_1}{L_1} \alpha~ {\rm because}~ K_1/L_1 \trianglelefteq {\rm Aut(}M_1/L_1{\rm )}]  \\& = \mathcal{O}(\alpha(\ol{v}))
        \end{split}
\end{equation*}
Therefore by universal property of quotient there exists $\widetilde{\alpha} : M_1/K_1 \to M_1/K_1$ such that the following diagram commutes.

\begin{figure}[h]
\begin{center}
\begin{tikzcd}
M_1/L_1 \arrow[r, "\alpha"] \arrow[d, "p"]                                      & M_1/L_1 \arrow[d, "p"]                      \\
\frac{M_1/L_1}{K_1/L_1} \arrow[r, "\widetilde{\alpha}", dashed] & \frac{M_1/L_1}{K_1/L_1}

\end{tikzcd}

\caption{{Diagram }}\label{diagram1}
\end{center}
\end{figure}
Now we have to show that $\widetilde{\alpha}$ is an automorphism.
Clearly $\widetilde{\alpha}$ is onto. Let $\ol{v_1}, \ol{v_2} \in Y.$ Suppose $\mathcal{O}(\ol{v_1})$ and $\mathcal{O}(\ol{v_2})$
be $K_1/L_1$ orbits of $Y$. Now,
\begin{equation*}
    \begin{split}
        \widetilde{\alpha}(\mathcal{O}(\ol{v_1})) = \widetilde{\alpha}(\mathcal{O}(\ol{v_2})) & \implies \mathcal{O}(\alpha(\ol{v_1})) = \mathcal{O}(\alpha(\ol{v_2})) \\ & \implies \exists ~\omega \in K_1/L_1 ~{\rm such~ that }~ \omega\alpha(\ol{v_1}) = \alpha(\ol{v_2}) \\&\implies \alpha^{-1}\omega \alpha (\ol{v_1}) = \ol{v_2} \\& \implies \mathcal{O}(\ol{v_1}) = \mathcal{O}(\ol{v_2}) ~[\alpha^{-1}\omega \alpha \in K_1/L_1 ~{\rm since}~ K_1/L_1 \trianglelefteq {\rm Aut}(M_1/L_1)]
    \end{split}
\end{equation*}
Therefore, $\widetilde{\alpha}$ is one-one. Now, by the commutativity of the diagram and using the fact that $p$ is a covering map one can see that $\widetilde{\alpha}$ takes vertices to vertices, edges to edges, faces to faces. It also preserves incidence relations. Let $v\in $ Domain of $\widetilde{\alpha}$. Since $p$ is a covering map there exists a neighbourhood $N$ of $v$ which is evenly covered by $p$. Let $U$ be a component of $p^{-1}(N)$. Then $p:U\to N$ is a homeomorphism. Therefore $(p \circ \alpha)|_U = \widetilde{\alpha}|_N$. As $p$ and $\alpha$ both are continuous so is $\widetilde{\alpha}|_N$. Thus $\widetilde{\alpha}$ is continuous. 
Now, replacing $\alpha$ by $\alpha^{-1}$ we get $\widetilde{\beta}$ in place of $\widetilde{\alpha}$. $\widetilde{\beta}$ has same properties as of $\widetilde{\alpha}$. Now, $\widetilde{\alpha} \circ \widetilde{\beta} = id_{M_1/K_1} = \widetilde{\beta} \circ \widetilde{\alpha}$. Therefore $\widetilde{\alpha}^{-1} = \widetilde{\beta}$. So $\widetilde{\alpha}$ is a homeomorphism.
Hence $\widetilde{\alpha}$ is an automorphism of $M_1/K_1$.

\begin{claim}\label{clm5}
If $\alpha \in$ Aut$(M_1/L_1) \setminus \frac{G_1'}{L_1}$ then $\alpha(\mathcal{O}) = \mathcal{O}$ for all $\frac{G_1'}{L_1}$-orbits $\mathcal{O}$ of $M_1/L_1$.
\end{claim}

Let $\alpha \in $Aut$(M_1/L_1) \setminus \frac{G_1'}{L_1}$ and $\widetilde{\alpha}$ be the induced automorphism on $M_1/K_1$ as in Fig. \ref{diagram1}. Suppose, $\mathcal{O}_1$ and $\mathcal{O}_2$ be two $G_1'/L_1$-orbits of $M_1/L_1$. Let $a_1, a_2 \in M_1$ be such that $L_1a_1 \in \mathcal{O}_1$ and $L_1a_2 \in \mathcal{O}_2$ and $\alpha(L_1a_1) = L_1a_2$.
Since, $p(L_1a_i) = K_1a_i$ by commutativity of the diagram in Claim \ref{diag} we get $\widetilde{\alpha}(K_1a_1) = K_1a_2$. 
As $\widetilde\alpha$ does not take an element of $G_1/K_1$-orbit to an element of some other orbit so $K_1a_1$ and $K_1a_2$ belong to same $G_1/K_1$-orbit of $M_1/K_1$.
Therefore, there exists $gK_1 \in G_1/K_1$ such that $(gK_1)(K_1a_1) = K_1a_2 $.\\
Now, Since $(gK_1)(K_1a_1) = K_1(ga_1)$ thus $(gK_1)(K_1a_1) = K_1a_2 \implies K_1(ga_1) = K_1a_2 \implies \exists ~k \in K_1$ such that $(k \circ g)(a_1) = a_2$.\\
Let $g':= k \circ g \in G_1'$ then $g'(a_1) = a_2.$
Consider $g'L_1 \in G_1'/L_1$. Then $(g'L_1)(L_1a_1) = L_1a_2.$
This contradicts our assumption that $L_1a_1$ and $L_1a_2$ belong to two different $G_1'/L_1$-orbit of $M_1/L_1.$ This proves  Claim \ref{clm5}.\\
Clearly $Y$ is a toroidal cover of $X$ and from Claim \ref{clm5} it follows that $E(Y)$ has two Aut$(Y)\mbox{-}$orbits.\\
In the same way we can prove existence of a $2$ orbital cover for edge-homogeneous maps with edge symbol $(3,6;4,4), (6,6;3,3)$ and $(4,4;3,6)$.\\
By Prop. \ref{datta2020} we can conclude that there exists edge transitive covers of edge-homogeneous toroidal map. This completes the proof of Theorem \ref{thm-main1}.
\end{proof}

Let $X^\#$ denotes the associated equivelar map corresponding to semi-equivelar map $X$. In Figures \ref{fig-E_9} to \ref{fig-E_6} we have shown equivelar tilings corresponding to semi-equivelar tiling in red color. If $X = E/K$ then $X^\#$ is defined by $X^\# = E^\#/K$ where $E^\#$ is the associated equivelar tiling on the plane associated to $E$. Except maps of type $[3^2,4^1,3^1,4^1]$ and $[3^4,6^1]$ the automorphism groups of the tiling $E$ an $E^\#$ are same. Hence for these type of maps
$X$ and $X^\#$ has same automorphism group namely ${\rm Nor_{Aut(E)}}(K)/K$. We know that if $E$ be a tessellation of the plane then Aut($E$) is of the form $H \rtimes S$ where $H$ be the translation group and $S$ be the stabilizer of a point which we can take as origin. Further suppose $E/K$ be a toroidal map. Every translation of $H$ induces an automorphism of $E/K$. Let $\chi$ be the map obtained by taking $180$ degree rotation with respect to origin of the plane. Then Aut($E/K$) will be of the form $(H \rtimes K')/K$ where $K' \le S$ containing $\chi$. For the tiling of type $[4^4]$ suppose $R_1$ and $R_2$ denoted reflection about the lines $OA$ and $OA_5$, see Figure 1(V). Then possibility of $K'$ upto conjugacy will be $S, \langle \chi \rangle, \langle \chi, R_1 \rangle, \langle \chi, R_2 \rangle, \langle \chi, R_1R_2 \rangle$. For tiling of type $[3^6]$ let $R_1'$ and $R_2'$ denotes the reflections about $OA_1$ and $OA$ respectively. Then possibility of $K'$ up to conjugacy will be $S$, $\langle \chi \rangle$, $ \langle \chi, R_2 \rangle, \langle \chi, R_1R_2 \rangle$. Now observe that except the tilings of type $[3^2,4^1,3^1,4^1]$ and $[3^4,6^1]$ the stabilizer of the origin for $E$ and $E^\#$ are same. For type $[3^2,4^1,3^1,4^1]$ although $R_2$ is a symmetry of $E^\#$ but it is not a symmetry of $E$. Although $R_1$ is a symmetry of $E$ and $E^\#$. For tiling of type $[3^4,6^1]$ nither $R_1'$ nor $R_2'$ are symmetry of $E$ though they are symmetry of $E^\#$. With this observations in mind we will prove following series of lemmas and use them to prove Theorem  \ref{edge-no-of-orbits} and \ref{edge-thm-main1}.

\begin{lemma}\label{uni}
Let $E$ be a semi\mbox{-}equivelar tiling on the plane. Suppose $E$ has $m$ edge\mbox{-}orbits. Then 
{\rm (a)} If the type of $E$ is  $[3^1, 12^2]$, $[3^4,6^1]$, $[3^1,4^1,6^1,4^1]$ or $[4^1,8^2]$  then $m = 2$.\\
{\rm (b)} If the type of $E$ is  $[4^1,6^1, 12^1]$,$ [3^2,4^1,3^1,4^1]$ or $[3^3,4^2]$ then $m = 3$.
\end{lemma}
\begin{proof}
Using rotational, reflectional and transnational symmetries of the tillings in Example \ref{exm2} one can count the number of edge orbits.  
\end{proof}

\begin{lemma}\label{X-9}
Let $X_9 = E_{9}/K_9$ be a semi-equivelar toroidal map of type $[3^1,4^1,6^1,4^1]$. Then $X_9^\#$ has $m_9$ edge orbits if and only if $X_9$ has $2m_9$ many edge orbits.
\end{lemma}
\begin{proof}
Here possible values of $m_9$ will be $1$, $2$ and $3$. Let $G_9 = \langle \alpha_9, \beta_9, \chi_9 \rangle$. Where $\alpha_9 : z \mapsto z+A_9$, $\beta_9 : z \mapsto z+B_9$ and $\chi_9$ be the $180$ degree rotation about origin, see Figure \ref{fig-E_9}. $E_9^\#$ is of type $[3^6]$.
First suppose $m_9 = 3$.
Let $X_9^\#$ has $3$ edge orbits. ${\rm Aut}(X_9^\#) = {\rm Nor_{Aut(E_9^\#)}}(K_9)/K_9 ={\rm Nor_{Aut(E_9)}}(K_9)/K_9 = {\rm Aut}(X_9) $. Now, $G_9 \le {\rm Nor_{Aut(E_9^\#)}}(K_9)$. Action of $G_9$ on $E(E_9^\#)$ also gives $3$ orbits. Hence $X_9^\#$ to be $3$ orbital we must have ${\rm Nor_{Aut(E_9^\#)}}(K_9) = G_9$ or some conjugate of $G_9$. 
Now under the action of $G_9$, $E(E_9)$ has $6$ orbits. Symmetries of $E_9$ which fixes origin are also symmetries of $E_9^\#$. Hence $E(E_9)$ has $6$ ${\rm Nor_{Aut(E_9)}}(K_9)$-orbits. Thus $X_9$ is $6$ orbital. \\
Now, let $m_9 = 2$. Then Aut($X_9^\#$) is of the form $(H_9 \rtimes K')/K_9$ where $K'$ is conjugate to $\langle \chi_9,R_2 \rangle$ or $\langle \chi_9, R_1R_2 \rangle$. Since $m_9=2$ $K'$ is conjugate to $\langle \chi_9,R_2 \rangle$. One can see that under action of these groups $E(X_9)$ has $4$ orbits. Thus $X_9$ is $4$-orbital map.\\
Now, let $m_9 = 1$. Then Aut($X_9^\#$) is of the form $(H_9 \rtimes K')/K_9$ where $K'$ is conjugate to $\langle \chi_9,R_2R_1 \rangle$. One can see that under action of this group $E(X_9)$ has $2$ orbits. Thus $X_9$ is $2$-orbital map.\\
Conversely, let $X_9$ is $6$ orbital. Then $G_9/K_9 \le $ Aut($X_9$). These symmetries are also present in ${\rm Aut}(E_9^\#)$ and the group $G_9$ gives $3$ orbits on $E(E_9^\#)$. Since, Aut($X_9$) does not change $G_9/K_9$-orbits of $E(X_9)$ so Aut($X_9^\#$) will also not change $G_9/K_9$-orbits of $E(X_9^\#)$. Thus $X_9^\#$ is $3$-orbital.
Now suppose $X_9$ is $4$ orbital. Then its automorphism group will contain either $R_1$ or $R_2$ along with $G_9$. With these symmetries $E(X_9^\#)$ will have two orbits. Hence $X_9^\#$ is $2$-orbital.\\
Now suppose $X_9$ is $2$ orbital. Then its automorphism group will contain either $R_1$ and $R_2$ along with $G_9$. With these symmetries $E(X_9^\#)$ will have one orbit. Hence $X_9^\#$ is $1$-orbital. This completes the proof of Lemma \ref{X-9}.
\end{proof}

\begin{lemma}\label{X-8}
$X_8 = E_{8}/K_8$ be a semi-equivelar toroidal map of type $[3^1,12^2]$. Then $X_8^\#$ has $m_8$ edge orbits if and only if $X_8$ has $2m_8$ many edge orbits.
\end{lemma}
\begin{proof}
Here also the possibilities of $m_8$ are $1,2$ and $3$.  Let $G_8 = \langle \alpha_8, \beta_8, \chi_8 \rangle$. Where $\alpha_8 : z \mapsto z+A_8$, $\beta_8 : z \mapsto z+B_8$ and $\chi_8$ be the $180$ degree rotation about origin, see Figure \ref{fig-E_8}.
Let us assume $m_8=3$. Let $X_8^\#$ has $3$ edge orbits. Now proceeding in the same way as in Lemma \ref{X-9} we get ${\rm Nor_{Aut(E_8^\#)}}(K_8) = G_8$ or some conjugate of $G_8$. Now, $E(E_8)$ has $6$ $G_8$-orbits. Hence $E(X_8) = E(E_8/K_8)$ has $6$ $G_8/K_8$-orbits. Thus $X_8$ is $6$ orbital.
Now, let $m_8 = 2$. Then Aut($X_8^\#$) is of the form $(H_8 \rtimes K')/K_8$ where $K'$ is conjugate to $\langle \chi_8,R_2 \rangle$ or $\langle \chi_8, R_1R_2 \rangle$. Since $m_8=2$ $K'$ is conjugate to $\langle \chi_8,R_2 \rangle$. Now, one can see that under action og this group $E(E_8)$ gives $4$ orbits. Hence in this case $X_8$ is $4$ orbital.\\
Now, let $m_8 = 1$. Then Aut($X_8^\#$) is of the form $(H_8 \rtimes K')/K_8$ where $K'$ is conjugate to $\langle \chi_8,R_2R_1 \rangle$. One can see that under action of this group $E(X_8)$ has $2$ orbits. Thus $X_8$ is $2$-orbital map.\\
Conversely, let $X_8$ is $6$ orbital. Then $G_8/K_8 \le $ Aut($X_8$). Since these symmetries in $G_8$ are also present in ${\rm Aut}(E_8^\#)$ and the group $G_8$ gives $3$ orbits on $E(E_8^\#)$. Since, Aut($X_8$) does not change $G_8/K_8$-orbits of $E(X_8)$ so Aut($X_8^\#$) will also not change $G_8/K_8$-orbits of $E(X_8^\#)$. Thus $X_8^\#$ is $3$-orbital. Other two cases are similar as in Lemma \ref{X-9}.
\end{proof}

\begin{lemma}\label{X-5}
$X_5 = E_{5}/K_5$ be a semi-equivelar toroidal map of type $[3^2,4^1,3^1,4^1]$. If $X_5^\#$ has $m_5$ edge orbits then $X_5$ has $3m_5$ many edge orbits for $m_5 = 2$. Further if $X_5^\#$ is $1$ orbital then $X_5$ is either $6$ orbital or $3$ orbital. 
\end{lemma}
\begin{proof}
Here possible values of $m_5$ will be $1$ and $2$. Let $G_5 = \langle \alpha_5, \beta_5, \chi_5 \rangle$. Where $\alpha_5 : z \mapsto z+A_5$, $\beta_5 : z \mapsto z+B_5$ and $\chi_5$ be the $180$ degree rotation about origin, see Figure \ref{fig-E_5}. 
First suppose $m_5 = 2$.
Let $X_5^\#$ has $2$ edge orbits. Now, $G_5 \le {\rm Nor_{Aut(E_5^\#)}}(K_5)$. Action of $G_5$ on $E(E_5^\#)$ also gives $2$ orbits. Hence $X_5^\#$ to be $2$ orbital we must have ${\rm Nor_{Aut(E_5^\#)}}(K_5) = G_5$ or some conjugate of $G_5$. 
Now under the action of $G_5$, $E(E_5)$ has $6$ orbits. Aut($E_5$)$\le$Aut($E_5^\#$). So ${\rm Nor_{Aut(E_5)}}(K_5) \le {\rm Nor_{Aut(E_5^\#)}}(K_5) = G_5$. Elements of $G_5$ are also symmetry of $E_5$. Hence ${\rm Nor_{Aut (E_5)}}(K_5)=G_5$.  Hence $X_5$ has $6$ ${\rm Nor_{Aut(E_5)}}(K_5)$-orbits. Thus $X_5$ is $6$ orbital. \\
Now, let $m_5 = 1$. Then Aut($X_5^\#$) is of the form $(H_5 \rtimes K')/K_5$ where $K'$ is conjugate to $\langle \chi_5,R_1 \rangle$ or $\langle \chi_6, R_2 \rangle$ or $\langle \chi_6, R_1R_2 \rangle$. As, $m_5=1$ so $K'$ is conjugate to either $\langle \chi_6, R_2 \rangle$ or $\langle \chi_6, R_1R_2 \rangle$. Now $R_1R_2$ is $90$ degree rotation with respect to origin. It is not a symmetry of $E_5$.
So if $K'$ is conjugate to $\langle \chi_6, R_1R_2 \rangle$ then under the induced group $E(E_5)$ has $6$ orbits. So then $X_5$ becomes $6$ orbital. If $K'$ is conjugate to  $\langle \chi_6, R_2 \rangle$ then under action of this group $E(E_5)$ has $3$ orbits. Thus $X_5$ becomes $3$ orbital.
\end{proof}

\begin{lemma}\label{X-11}
$X_{11} = E_{11}/K_{11}$ be a semi-equivelar toroidal map of type $[4^1,6^1,12^1]$. Then $X_{11}^\#$ has $m_{11}$ edge orbits if and only if $X_{11}$ has $3m_5$ many edge orbits for $m_5 = 1$ and $2$. Further if $X_{11}^\#$ is $3$ orbital if and only if $X_{11}$ is $12$ orbital.
\end{lemma}
\begin{proof}
Here possible values of $m_{11}$ will be $1$, $2$ and $3$. Let $G_{11} = \langle \alpha_{11}, \beta_{11}, \chi_{11} \rangle$. Where $\alpha_{11} : z \mapsto z+A_{11}$, $\beta_{11} : z \mapsto z+B_{11}$ and $\chi_{11}$ be the $180$ degree rotation about origin, see Figure \ref{fig-E_11}. $E_{11}^\#$ is of type $[3^6]$.
First suppose $m_{11} = 3$.
Let $X_{11}^\#$ has $3$ edge orbits. ${\rm Aut}(X_{11}^\#) = {\rm Nor_{Aut(E_{11}^\#)}}(K_{11})/K_{11} ={\rm Nor_{Aut(E_{11})}}(K_{11})/K_{11} = {\rm Aut}(X_{11}) $. Now, $G_{11} \le {\rm Nor_{Aut(E_{11}^\#)}}(K_{11})$. Action of $G_{11}$ on $E(E_{11}^\#)$ also gives $3$ orbits. Hence $X_{11}^\#$ to be $3$ orbital we must have ${\rm Nor_{Aut(E_{11}^\#)}}(K_{11}) = G_{11}$ or some conjugate of $G_{11}$. 
Now under the action of $G_{11}$, $E(E_{11})$ has $12$ orbits. Symmetries of $E_{11}$ which fixes origin are also symmetries of $E_{11}^\#$. Hence $X_{11}$ has $12$ ${\rm Nor_{Aut(E_{11})}}(K_{11})$-orbits. Thus $E(X_{11})$ is $12$ orbital. \\
Now, let $m_{11} = 2$. Then Aut($X_{11}^\#$) is of the form $(H_{11} \rtimes K')/K_{11}$ where $K'$ is conjugate to $\langle \chi_{11},R_2 \rangle$ or $\langle \chi_{11}, R_1R_2 \rangle$. Since $m_{11}=2$ $K'$ is conjugate to $\langle \chi_{11},R_2 \rangle$. One can see that under action of this group $E(X_{11})$ has $6$ orbits. Thus $X_{11}$ is $6$-orbital map.\\
Now, let $m_{11} = 1$. Then Aut($X_{11}^\#$) is of the form $(H_{11} \rtimes K')/K_{11}$ where $K'$ is conjugate to $\langle \chi_{11},R_2R_1 \rangle$. One can see that under action of this group $E(X_{11})$ has $3$ orbits. Thus $X_{11}$ is $3$-orbital map.\\
Conversely, let $X_{11}$ is $12$ orbital. Then $G_{11}/K_{11} \le $ Aut($X_{11}$). Since these symmetries are also present in ${\rm Aut}(E_{11})$ and the group $G_{11}$ gives $3$ orbits on $E(E_{11}^\#)$. Since, Aut($X_{11}$) does not change $G_{11}/K_{11}$-orbits of $E(X_{11})$ so Aut($X_{11}^\#$) will also not change $G_{11}/K_{11}$-orbits of $E(X_{11}^\#)$. Thus $E(X_{11}^\#)$ is $3$-orbital.
Now suppose $X_{11}$ is $6$ orbital. Then its automorphism group will contain either $R_1$ or $R_2$ along with $G_{11}$. With these symmetries $E(X_{11}^\#)$ will have two orbit.\\
Now suppose $X_{11}$ is $3$ orbital. Then its automorphism group will contain either $R_1$ and $R_2$ along with $G_{11}$. With these symmetries $E(X_{11}^\#)$ will have one orbit. This completes the proof of Lemma \ref{X-11}.
\end{proof}

\begin{lemma}\label{X-6}
$X_6 = E_{6}/K_6$ is a semi-equivelar toroidal map of type $[4^1,8^2]$. Then 
{\rm (1)}$X_6^\#$ is $2$ orbital implies $X_6$ is either $3$ or $4$ orbital.
{\rm (2)}$X_6^\#$ is $1$ orbital implies $X_6$ is either $2$ or $3$ orbital.
{\rm (3)}$X_6$ is $4$ orbital implies $X_6^\#$ is $2$ orbital.
{\rm (4)}$X_6$ is $2$ orbital implies $X_6^\#$ is $1$ orbital.
\end{lemma}
\begin{proof}
Let $X_6 = E_{6}/K_6$ is a semi-equivelar toroidal map of type $[4^1,8^2]$. Let $G_6 = \langle \alpha_6, \beta_6, \chi_6 \rangle$. Where $\alpha_6 : z \mapsto z+A_6$, $\beta_6 : z \mapsto z+B_6$ and $\chi_6$ be the $180$ degree rotation about origin, see Figure \ref{fig-E_6}. 
Let $X_6^\#$ has $2$ edge orbits. ${\rm Aut}(X_6^\#) = {\rm Nor_{Aut(E_6^\#)}}(K_6)/K_6 = {\rm Nor_{Aut(E_6)}}(K_6)/K_6 = {\rm Aut}(X_6)$. Now, $G_6 \le {\rm Nor_{Aut(E_6^\#)}}(K_6)$. Action of $G_6$ on $E(E_6^\#)$ gives $2$ orbits. Hence $X_6^\#$ to be $2$ orbital we must have ${\rm Nor_{Aut(E_6^\#)}}(K_6)$ is conjugate to $G_6$ or $G_6'\langle \alpha_6,\beta_6,\chi_6,R_1\rangle$. 
Now under the action of $G_6$, $E(E_6)$ has $4$ orbits and under action of $G_6'$ $E(E_6)$ gives $3$ orbits. Symmetries of $E_6$ which fixes origin are also symmetries of $E_6^\#$.Thus $X_6$ is $3$ or $4$ orbital. \\
Now, let $X_6^\#$ has $1$ edge orbits. Then Aut($X_6^\#$) is of the form $(H_6 \rtimes K')/K_6$ where $K'$ is conjugate to $\langle \chi_6, R_2 \rangle$ or $\langle \chi_6, R_1R_2 \rangle$. One can see that under action of these groups $E(X_6)$ has $3$ and $2$ orbits. Thus $X_6$ is $2$ or $3$-orbital map.\\
Conversely, let $X_6$ is $4$ orbital. Then $G_6/K_6 \le $ Aut($X_6$). Since these symmetries are also present in Aut($E_6^\#$) the group $G_6$ gives $2$ orbits on $E(E_6^\#)$. As ${\rm Aut}(X_6)$ does not change $G_6/K_6$-orbits of $E(X_6)$ so Aut($X_6^\#$) will also not change $G_6/K_6$-orbits of $E(X_6^\#)$. Thus $X_6^\#$ is $2$-orbital.\\
Now suppose $X_6$ is $2$ orbital. Then its automorphism group will contains $R_1R_2$  along with $G_6$. With these symmetries $E(X_6^\#)$ will have one orbit. This completes the proof of Lemma \ref{X-6}.
\end{proof}

\begin{lemma}\label{X-10}
$X_{10} = E_{10}/K_{10}$ is a semi-equivelar toroidal map of type $[3^4,6^1]$. Then
{\rm (1)} $X_{10}^\#$ is $3$ orbital then $X_{10}$ is $8$ orbital.
\end{lemma}
\begin{proof}
Let $X_{10} = E_{10}/K_{10}$ is a semi-equivelar toroidal map of type $[3^4,6^1]$. Let $G_{10} = \langle \alpha_{10}, \beta_{10}, \chi_{10} \rangle$. Where $\alpha_{10} : z \mapsto z+A_{10}$, $\beta_{10} : z \mapsto z+B_{10}$ and $\chi_{10}$ be the $180$ degree rotation about origin, see Figure \ref{fig-E_10}.
Let $X_{10}^\#$ be $3$ orbital. 
${\rm Aut}(X_{10}^\#) = {\rm Nor_{Aut(E_{10}^\#)}}(K_{10})/K_{10}.$ 
Now, $G_{10} \le {\rm Nor_{Aut(E_{10}^\#)}}(K_{10})$. Action of $G_{10}$ on $E(E_{10}^\#)$ gives $3$ orbits. 
Hence $X_{10}^\#$ to be $3$ orbital we must have ${\rm Nor_{Aut(E_{10}^\#)}}(K_{10})$ is conjugate to $G_6$. Note that Aut($E_{10}$) $\le$ Aut($E_{10}^\#$)
Now under the action of $G_{10}$, $E(E_{10})$ has $8$ orbits. Thus $X_{10}$ is $8$ orbital.
\end{proof}

\begin{proof}[Proof of Theorem \ref{edge-no-of-orbits}]
Let $X_8, X_5, X_9$ be a semi-equivelar toroidal map of type $[3^1,12^2],$ 
$[3^2,4^1,3^1,4^1], [3^1,4^1,6^1,4^1]$ respectively. Let $X_i = E_i/K_i$. Let $G_i$ be as in Lemmas \ref{X-8}, \ref{X-5}, \ref{X-9}. Then $E(E_i)$ has $6$ $G_i$-orbits. Therefore $E(X_i)=E(E_i/K_i)$ has $6$ $G_i/K_i$ orbits. Since $G_i/K_i \le$ Aut($X_i$) so number of edge orbits of $X_i$ is less than or equal to $6$. Now $X_i^\#$ is a map of type $[3^6]$ for i$=8,9$ and of type $[4^4]$ for i$=5$. From Lemmas \ref{X-8}, \ref{X-5}, \ref{X-9} we can conclude that if there exists a $3$-orbital map of type $[3^6]$ then the corresponding semi-equivelar maps has $6$ edge orbits. Existence of a $3$-orbital map of type $[3^6]$ and $2$ orbital map of type $[4^4]$ is guaranteed by Example \ref{example_bdd}. This proves part (a) of Theorem \ref{edge-no-of-orbits}.\\
Now, let $X_{6}$ be a semi-equivelar toroidal map of type $[4^1,8^2]$. Let $G_{6}$ be as in Lemma \ref{X-6}. Then $G_{6}$ gives $4$ orbits on $E(E_{6})$ Then by same reason as part (a) we get number of edge orbits are less than or equals to $4$ and this bound is also sharp by Lemma \ref{X-6} and Example \ref{example_bdd}. This proves part (b) of Theorem \ref{edge-no-of-orbits}.\\
Now, let $X_{10}$ be a semi-equivelar toroidal map of type $[3^4,6^1]$. Let $G_{10}$ be as in Lemma \ref{X-10}. Then $G_{10}$ gives $8$ orbits on $E(E_{10})$ Then by same reason as part (a) we get number of edge orbits are less than or equals to $8$ and this bound is also sharp by Lemma \ref{X-10} and Example \ref{example_bdd}. This proves part (c) of Theorem \ref{edge-no-of-orbits}.\\
Let $X_4 = E_4/K_4$ be a semi-equivelar toroidal map with vertex type $[3^3,4^2]$, where $K_4$ is a fixed point free discrete subgroup of Aut($E_4$). Since $K_4$ has no fixed points, $K_4$ consist of only translations and glide reflections. As $X_4$ is orientable, $K_4$ does not contain any glide reflections. Hence, $K_4$ contains only translations. Let $G_4 = \langle \alpha_4, \beta_4, \chi_4 \rangle$ where $\alpha_{4} : z \mapsto z+A_4$, $\beta_{4} : z \mapsto z+B_{4}$ and $\chi_{4}$ be the $180$ degree rotation about origin, see Figure \ref{fig-E_4}.
Now $E(E_4)$ has $3$ $G_4$-orbits. Therefore $E(X_4)=E(E_4/K_4)$ has $3$ $G_4/K_4$-orbits. $G_4/K_4 \le {\rm Aut}(X_4)$. Therefore number of Aut($X_4$)-orbits of $E(X_4)$ is less than or equals to $3$. 
Again $E(E_4)$ also has $3$ Aut($E_4$)-orbits. Hence number of Aut($X_4$)-orbits of $E(X_4)$ is grater than or equals to $3$. Thus $X_4$ is $3$ orbital. This completes the proof of part (d).\\
let $X_{11}$ be a semi-equivelar toroidal map of type $[4^1,6^1,12^1]$. Let $G_{11}$ be as in Lemma \ref{X-11}. Then $G_{11}$ gives $12$ orbits on $E(E_{11})$ Then by same reason as part (a) we get number of edge orbits are less than or equals to $12$ and this bound is also sharp by Lemma \ref{X-11} and Example \ref{example_bdd}. This proves part (e) of Theorem \ref{edge-no-of-orbits}.
\end{proof}

\begin{proof}[Proof of Theorem \ref{edge-thm-main1}]
Let $X_8$ be a semi-equivelar $m_8$-edge orbital toroidal map of type $[3^1,12^2]$. Let $X_8^\#$ be the associated equivelar map of type $[3^6]$. By Lemma \ref{X-8} we get $X_3^\#$ has $n_8:=m_8/2$ many edge  orbits. Now by Theorem \ref{thm-main1} we have covering $\eta_{k_8} : Y_{k_8}^\# \to X_8^\#$ where $Y_{k_8}^\#$ is a $k_8$-edge orbital for each $ k_8 \le 3$. Now, if we consider the map of type $[3^1,12^2]$ corresponding to the equivelar map $Y_{k_8}^\#$, say $Y_{k_8}$, then by Lemma \ref{X-8} it will be a $(2 \times k_3)$-edge orbital map. Clearly $Y_{k_8}$ is a cover of $X_8$. Hence for given $m_8$-orbital map of type $[3^1,12^2]$ there exists a $k_8$ orbital cover of it for each $k_8 \le m_8$ and $2 \mid k_8$.\\
The proof will go exactly same way for maps of type $[3^1,4^1,6^1,4^1]$ using Lemma \ref{X-9}. This proves part (a) of Theorem \ref{edge-thm-main1}.\\
Let $X_5$ be a $6$-orbital semi-equivelar toroidal map of type $[3^2,4^1,3^1,4^1]$. We can assume $X_5 = E_5/K_5$ for some discrete fixed point free subgroup of Aut($E_5$). Actually $K_5$ consist of only translations.
Let $G_5$ be defined as in Lemma \ref{X-5}. Then $E(X_5)$ has $6$ $G_5/K_5$-orbits. Let $G_5' = \langle \alpha_5, \beta_5, \chi_5, \rho_5 \rangle$ where $\rho_5$ be the map obtained by taking reflection of $E_5$ about the line passing through $O$ and $A$ (see Figure \ref{fig-E_5}). Observe that $E(E_5)$ has $3$ $G_5'$-orbits. Now proceeding in similar way as in Theorem \ref{thm-main1} we get existence of a $3$-orbital cover of $X_5$. This proves part (b) of Theorem \ref{edge-thm-main1}.\\
Let $X_{10}$ be a $8$-orbital semi-equivelar toroidal map of type $[3^4,6^1]$. We can assume $X_{10} = E_{10}/K_{10}$ for some discrete fixed point free subgroup of Aut($E_{10}$). Actually $K_{10}$ consist of only translations.
Let $G_{10}$ be defined as in Lemma \ref{X-10}. Then $E(X_{10})$ has $8$ $G_{10}/K_{10}$-orbits. Let $G_{10}' = \langle \alpha_{10}, \beta_{10}, \chi_{10}, \rho_{10} \rangle$ where $\rho_{10}$ be the map obtained by taking $60$ degree rotation of $E_{10}$ about $O$ (see Figure \ref{fig-E_10}). Observe that $E(E_{10})$ has $2$ $G_{10}'$-orbits. Now proceeding in similar way as in Theorem \ref{thm-main1} we get existence of a $2$-orbital cover of $X_{10}$. This proves part (c) of Theorem \ref{edge-thm-main1}.\\
Let $X_{11}$ be a $12$-orbital semi-equivelar toroidal map of type $[4^1,6^1,12^1]$. We can assume $X_{11} = E_{11}/K_{11}$ for some discrete fixed point free subgroup of Aut($E_{11}$). Actually $K_{11}$ consist of only translations.
Let $G_{11}$ be defined as in Lemma \ref{X-11}. Then $E(X_{11})$ has $12$ $G_{11}/K_{11}$-orbits. Let $G_{11}' = \langle \alpha_{11}, \beta_{11}, \chi_{11}, \rho_{11} \rangle$ where $\rho_{11}$ be the map obtained by taking reflection of $E_{11}$ about the line passing through $O$ and $A$. Observe that $E(E_{11})$ has $6$ $G_{11}'$-orbits. Now proceeding in similar way as in Theorem \ref{thm-main1} we get existence of a $6$-orbital cover $Y_{11}$ of $X_{11}$. Now consider $G_{11}''=\langle \alpha_{11}, \beta_{11}, \chi_{11}, \rho_{11}, \tau_{11} \rangle$ where $\tau_{11}$ is the map obtained by taking reflection of $E_{11}$ about the line passing through $O$ and $A_{11}$ (see Figure \ref{fig-E_11}). Observe that $E(E_{11})$ has $3$ $G_{11}''$-orbits. Now proceeding in similar way as in Theorem \ref{thm-main1} we get existence of a $3$-orbital cover $Z_{11}$ of $Y_{11}$. This proves part (d) of Theorem \ref{edge-thm-main1}.\\
Let $X_6$ be a $4$-orbital semi-equivelar toroidal map of type $[4^1,8^2]$. We can assume $X_6 = E_6/K_6$ for some discrete fixed point free subgroup $K_6$ of Aut($E_6$). Actually $K_6$ consist of only translations.
Let $G_6$ be defined as in Lemma \ref{X-6}. Then $E(X_6)$ has $4$ $G_6/K_6$-orbits. Let $G_6' = \langle \alpha_6, \beta_6, \chi_6, \rho_6 \rangle$ where $\rho_6$ be the map obtained by taking reflection of $E_6$ about the line passing through $O$ and $A$. Observe that $E(E_6)$ has $3$ $G_6'$-orbits. Now proceeding in similar way as in Theorem \ref{thm-main1} we get existence of a $3$-orbital cover $Y_6$ of $X_6$. Now consider $G_6''=\langle \alpha_6, \beta_6, \chi_6, \rho_6, \tau_6 \rangle$ where $\tau_6$ is the map obtained by taking reflection of $E_6$ about the line passing through $O$ and $A_6$ (see Figure \ref{fig-E_6}). Observe that $E(E_6)$ has $2$ $G_6''$-orbits. Now proceeding in similar way as in Theorem \ref{thm-main1} we get existence of a $2$-orbital cover $Z_6$ of $Y_6$. This proves part (e) of Theorem \ref{edge-thm-main1}.
\end{proof}

Now we are moving to see number of sheets of the covers obtained above. For that we make,

\begin{claim}\label{sheet} 
The cover $Y$ in Theorem \ref{thm-main1} is a $m^2$ sheeted covering of $X$.
\end{claim}

To do this we need following two results from the theory of covering spaces.

\begin{result}\label{result1} (\cite{AH2002})
Let $p:(\widetilde{X},\widetilde{x_0})\to (X,x_0)$ be a path-connected covering space of the path-connected, locally path-connected space $X$, and let $H$ be the subgroup $p_*(\pi_1(\widetilde{X},\widetilde{x_0})) \subset \pi_1(X,x_0).$ Then,
\begin{enumerate}
    \item[1.] This covering space is normal if and only if $H$ is a normal subgroup of $\pi_1(X,x_0)$
    \item[2.] $G(\widetilde{X})$(the group of deck transformation of the covering $\widetilde{X}\to X$) is isomorphic to $N(H)/H$ where $N(H)$ is the normalizer of $H$ in $\pi_1(X,x_0)$.
\end{enumerate}
In particular, $G(\widetilde{X})$ is isomorphic to $\pi_1(X,x_0)/H$ if $\widetilde{X}$ is a normal covering. Hence for universal cover $\widetilde{X}\to X$ we have $G(\widetilde{X}) \simeq \pi_1(X)$.
\end{result} 

\begin{result}\label{result2} (\cite{AH2002})
The number of sheets of a covering space $p:(\widetilde{X},\widetilde{x_0})\to (X,x_0)$ with $X$ and $\widetilde{X}$ path-connected equals the index of $p_*(\pi_1(\widetilde{X},\widetilde{x_0}))$ in $\pi_1(X,x_0)$.
\end{result}

In our situation applying Result \ref{result1} for the covering $E_i\to E_i/K_i$ we get $\pi_1(E_i/K_i) = K_i$. For the covering $E_i \to E_i/\mathcal{L}_i$ we get $\pi_1(E_i/\mathcal{L}_i) = \mathcal{L}_i$. Thus applying Result \ref{result2} we get number of sheets of $Y$ over $X$ is $=n:=[K_i:\mathcal{L}_i] = m^2$ for all $i = 3,4,5,6, 7$.  This proves our Claim \ref{sheet}.

\begin{proof}[Proof of Theorem \ref{thm-main2}]
Let $X$ be an edge-homogeneous map of type $(m, \ell, u,v)$. Then form Prop. \ref{propo-1} we get $X = M_i/K$ for some discrete subgroup $K$ of Aut($M_i$). Now $Y$ covers $X$ if and only if  $Y = M_i/L$ for some subgroup $L$ of $K$ generated by $2$ translations corresponding to $2$ independent vectors. Let $K = \langle \gamma , \delta \rangle$. Now consider $L_n = \langle \gamma^n, \delta \rangle$ and $Y_n = M_i/L_n$. then $Y_n$ covers $X$.
Number of sheets of the cover $Y_n \longrightarrow X$ is equal to $[K : L_n] = n.$ Hence $Y_n$ is our required $n$ sheeted cover of $X$.
\end{proof}
\begin{proof}[Proof of Theorem \ref{thm-main3}]
Here two maps are isomorphic if they are isomorphic as maps. Two maps are equal if the orbits of $\mathbb{R}^2$ under the action of corresponding groups are equal as sets.
Suppose $X$ and $K$ be as in the proof of Theorem \ref{thm-main2}. Let $n \in \mathbb{N}$. 
Let $Y =  E/L$ be $n$ sheeted cover of $X$. 
Let $L = \langle \omega_1, \omega_2\rangle$. Where $\omega_1, \omega_2 \in K = \langle \gamma, \delta \rangle$. Suppose $\omega_1 = \gamma^a \circ \delta^b$ and $\omega_2 = \gamma^c \circ \delta^d$ where $a,b,c,d \in \ZZ$. Define $M_Y = \begin{bmatrix} a & c\\b &d \end{bmatrix}$.
We represent $Y$ by the associated matrix $M_Y$. This matrix representation corresponding to a map is unique as $\gamma$ and $\delta$ are translations along two linearly independent vectors. Denote area of the torus $Y$ by $\Delta_Y$. As $Y$ is $n$ sheeted covering of $X$ so $\Delta_Y = n\Delta_X \implies$ area of the parallelogram spanned by $w_1$ and $w_2 = n \times$ area of the parallelogram spanned by $\gamma$ and $\delta$. That means $|det(M_Y)| = n$. Therefore for each $n$ sheeted covering, the associated matrix belongs to 
$$ \mathcal{S}:=\{ M \in GL(2,\ZZ): |det(M)| = n \}.$$ 
Conversely for every element of $\mathcal{S}$ we get a $n$ sheeted covering $Y$ of $X$ by associating $\begin{bmatrix} a & c\\b &d \end{bmatrix}$ to $E/\langle a\gamma + b\delta, c\gamma +d\delta\rangle$. So there is an one to one correspondence to $n$-sheeted covers of $X$ and $\mathcal{S}$. To proceed further we need following two lemmas.
\begin{lemma}\label{equl}
Let $Y_1$ and $Y_2$ be maps and $M_1$ and $M_2$ be associated matrix of them respectively. Then $Y_1 = Y_2$ if and only if  there exists an unimodular matrix $($an integer matrix with determinant $1$ or $-1)$ $U$ such that $M_1U = M_2.$  
\end{lemma}
\begin{proof}
Let $Y_1 = Y_2$. Let $i:Y_1\to Y_2$ be an isomorphism. We can extend  $i$ to $\widetilde{i} \in Aut(E)$. Then $\widetilde{i}$ will take fundamental parallelogram of $Y_1$ to that of $Y_2$. Hence the latices formed by $L_1$ and $L_2$ are same say $\Lambda$. $\widetilde{i}$ transforms $\Lambda$ to itself. Therefore from \cite{HW1979}(Theorem 32, Chapter 3) we get matrix of the transformation is unimodular. Our lemma follows from this.
\\
Conversely suppose $M_1U = M_2$ where $U$ is an unimodular matrix. Let $M_1 = (w_1~ w_2), M_2 = (w_1'~w_2') $ and $U = \begin{bmatrix} a & b\\c &d \end{bmatrix}$ where $w_i, w_i'$ are column vectors for $i=1,2$.
Therefore 
$$ M_2 = M_1U \implies (w_1'~ w_2') = (w_1~ w_2)\begin{bmatrix} a & b\\c & d \end{bmatrix} = (aw_1+cw_2~~bw_1+dw_2).$$
Now suppose $L_1 = \langle \alpha_1, \beta_1 \rangle$ and $L_2 = \langle \alpha_2, \beta_2 \rangle$ and $A_i, B_i$ be the vectors by which $\alpha_i$ and $\beta_i$ translating the plane for $i=1,2$ and let $C$ and $D$ be the vectors corresponding to $\gamma$ and $\delta$.
Let $$A_1 = p_1C + q_1D, B_1 = s_1C + t_1D, A_2 = p_2C + q_2D, B_2 = s_2C + t_2D.$$
Now $w_1' = \begin{pmatrix} p_2 \\ q_2\end{pmatrix} = a\begin{pmatrix} p_1 \\ q_1\end{pmatrix} + c\begin{pmatrix} s_1 \\ t_1\end{pmatrix} = \begin{pmatrix} ap_1+cs_1 \\ aq_1+ct_1\end{pmatrix}$.
Therefore 
\begin{equation*}
    \begin{split}
        A_2 &= (ap_1 + cs_1)C + (aq_1+ct_1)D \\&= a(p_1C + q_1D) + c(s_1C + t_1D) \\&= aA_1 + cB_1
    \end{split}
\end{equation*}
Hence $\alpha_2 \in L_1$. Similarly $\beta_2 \in L_1$. Therefore $L_2 \le L_1$. Proceeding in the similar way and using the fact that $det(U) = \pm 1$ we get $L_1 \le L_2$.
Therefore $L_1 = L_2$. Thus $Y_1 = E/L_1 = E/L_2 = Y_2$. This completes the proof of Lemma \ref{equl}.
\end{proof}
\begin{lemma}\label{isomm}
Let $Y_1$ and $Y_2$ be two toroidal maps with associated matrix $M_1$ and $M_2$ respectively. Then $Y_1 \simeq Y_2$ if and only if  there exists $A \in G_0$ and $B\in GL(2,\ZZ)$ such that $M_1 = AM_2B$ where $G_0$ is group of rotations and reflections fixing the origin in $E$.
\end{lemma}
\begin{proof}
Let $Y_1 \simeq Y_2$ and $\alpha : Y_1 \to Y_2$ be an isomorphism. Now $\alpha$ can be extended to an automorphism of the covering plane $E$, call that extension be $\widetilde{\alpha}$. 
Clearly $\widetilde{\alpha}$ will take fundamental parallelogram of $Y_1$ to that of $Y_2$. Now the only ways to transform one fundamental region to another are rotation, reflection and change of basis of $E$. Multiplication by an element of $GL(2,\ZZ) $ will take care of base change. Rotation, reflection or their composition will take care by multiplication by $A\in G_0$. Hence we get $M_1 = AM_2B$.\\
Conversely let $M_1 = AM_2B$. $A \in G_0$ so the combinatorial type of the torus associated to the matrix $AM_2$ and $M_2$ are same. Geometrically multiplying by elements of $GL(2,\ZZ)$ corresponds to modifying the fundamental domain by changing the basis. Hence this will not change the combinatorial type of the torus. Thus $Y_1 \simeq Y_2$. This completes the proof of Lemma \ref{isomm}.
\end{proof}
Now define a relation on $\mathcal{S}$ by $P\sim Q \iff P= QU$ for some unimodular matrix $U$. Clearly this is an equivalence relation. Consider $\mathcal{S}' = \mathcal{S}/\sim $. So by Lemma \ref{equl} we can conclude that there are $\#\mathcal{S}'$ many distinct $n$ sheeted cover of $X$ exists. Let's find this cardinality.
Now for every $m\times n$ matrix $P$ with integer entries has an unique $m \times n$ matrix $H$, called hermite normal form of $P$, such that $H = PU$ for some unimodular matrix $U$. All elements of an equivalence class of $\mathcal{S}'$ has  same hermite normal form and we take this matrix in hermite normal form as representative of that equivalence class. Thus to find cardinality of $\mathcal{S}'$ it is enough to find number of distinct matrices $M$ which are in hermite normal form and has determinant $n$. We do not take the matrices with determinant $-n$ because by multiplying by the unimodular matrix  
$\begin{bmatrix} 0 & 1\\1 &0 \end{bmatrix}$ changes sign of the determinant.
As $M$ is in lower triangular form so take $M = \begin{bmatrix} a & 0\\b &d \end{bmatrix}$.
Then $det(M) = ad = n \implies a= n/d.$ By definition of hermite normal form $b\geq0$ and $b<d$ so $b$ has $d$ choices for each $d|n.$ Hence there are precisely $\sigma(n):=\sum_{d|n}d$ many distinct $M$ possible. Thus $\#\mathcal{S}' = \sigma(n).$
Let 
$\mathcal{S}_1 = \{ M| M $ is a representative of an equivalence class of $ \mathcal{S}'$ which is in hermite normal form$\}$
Clearly $\#\mathcal{S}_1 = \sigma(n).$ Now define a relation on $\mathcal{S}_1 $ by $M_1\sim M_2  \iff \exists A \in G_0 ~such ~that~ M_1 = AM_2$. Clearly this is an equivalence relation. 
Consider $\mathcal{S}_2 := \mathcal{S}_1/\sim$.
By Lemma \ref{isomm} it follows that there are $\#\mathcal{S}_2$ many $n$ sheeted covers upto isomorphism. Because here all matrices $M_i$ has same determinant so $M_1 = AM_2B \implies det(A)det(B) = 1$. As det($A$) and det($B$) both are integer so they belongs to $\{1,-1\}$ i.e. they are unimodular matrices.
Now we have to find $\#\mathcal{S}_2$.  Observe that the matrix representation of elements of $G_0$ with respect to the basis $\{\alpha_i(0), \beta_i(0)\}$ have integer entries because lattice points must go to lattice points by a symmetry of the plane where $\alpha_i : z \mapsto z+A_i$ and $\beta_i:z\mapsto z+B_i$ are two translations of $E_i$ (see Figure  \ref{fig:eh} to \ref{fig-E_6}). Suppose $M_1$ and $M_2 \in \mathcal{S}_2$ such that $M_1 \sim M_2$. So there exists $A \in G_0$ such that $M_1 = AM_2$. Let 
$M_1 = \begin{bmatrix} \frac{n}{d_1} & 0\\c_1 &d_1 \end{bmatrix}$ , 
$M_2 = \begin{bmatrix} \frac{n}{d_2} & 0\\c_2 &d_2 \end{bmatrix}$ and 
$A = \begin{bmatrix} p & q\\r &s \end{bmatrix}$. 
Then 
\begin{equation}\label{mat}
    \begin{split}
        M_1 = AM_2 &\implies  \begin{bmatrix} \frac{n}{d_1} & 0\\c_1 &d_1 \end{bmatrix} = \begin{bmatrix} p & q\\r &s \end{bmatrix} \begin{bmatrix} \frac{n}{d_2} & 0\\c_2 &d_2 \end{bmatrix} = \begin{bmatrix} \frac{np}{d_2}+qc_2 & qd_2\\\frac{rn}{d_2}+sc_2 &sd_2 \end{bmatrix} \\
        &\implies qd_2 = 0 \\&\implies q=0 ~since ~d_2 \neq 0.
    \end{split}
\end{equation}
Therefore $A = \begin{bmatrix} p & 0\\r &s \end{bmatrix}$. $det(A) = 1 \implies ps = 1 \implies s= 1/p.$
Again from equation \ref{mat} we get $$np/d_2 = n/d_1 \implies p = d_2/d_1$$ and $$rn/d_2 + c_2/p = c_1 \implies r = (d_2c_1-d_1c_2)/n$$
Therefore $$A = \begin{bmatrix} d_2/d_1 & 0\\(d_2c_1-d_1c_2)/n &d_1/d_2 \end{bmatrix}.$$ As $A$ has integer entries and $d_1, d_2$ are positive so $d_1 = d_2 = d$(say) and $n|d(c_1-c_2)$.
Hence $$A = \begin{bmatrix} 1 & 0\\ d(c_1-c_2)/n & 1 \end{bmatrix}.$$
Now $A \in G_0$ and $G_0 = D_6$ for maps with edge symbol $(3,3;6,6),(6,6;3,3), (3,6;4,4),$ and  $(4,4;3,6)$ and maps of type $[3^1,4^1,6^1,4^1], [3^1,12^2], [4^1,6^1,12^1]$; $G_0 = D_4$ for maps of type $(4,4;4,4),[3^2,4^1,3^1,4^1],[4^1,8^2]$; $G_0=\ZZ_6$ for $[3^4,6^1]$.
Here $D_6$ is generated by $\begin{bmatrix} 0 & -1\\ 1 & 1 \end{bmatrix}$ and $\begin{bmatrix} -1 & -1\\ 0 & 1 \end{bmatrix}$. $D_4$ is generated by $\begin{bmatrix} 0 & 1\\ -1 & 0 \end{bmatrix}$ and $\begin{bmatrix} -1 & 0\\ 0 & 1 \end{bmatrix}$. $\ZZ_6$ is generated by $\begin{bmatrix} 0 & -1\\ 1 & 1 \end{bmatrix}$.
One can check that only matrices in $G_0$ having diagonal entries $1$ is identity matrix. Hence $A = I_2$. Therefore $c_1 = c_2 \implies M_1 = M_2.$ Each equivalence class of $\mathcal{S}_2$ is singleton. Therefore $\#\mathcal{S}_2 = \#\mathcal{S}_1 = \sigma(n)$. This proves  Theorem \ref{thm-main3}.
\end{proof}

\begin{proof}[Proof of Theorem \ref{thm-main4}]
Let $X$ be a $m$-orbital map of vertex type $(m,  \ell; u,v)$. Let $Y_1$ be a $k$-orbital cover of $X$. Consider number of sheets of the cover $Y_1\longrightarrow X$ be $n_1$. Let the set $C_1$ containing all $n$ sheeted covering of $X$ for $n\le n_1-1$. Now check that does there exists a $k$-orbital cover or not in $C_1$. If there does not exists one, then $Y_1$ be a minimal $k$-orbital cover otherwise take $Y_2$ be a $k$-orbital cover in $C_1$. Let number of sheets for the covering $Y_2\longrightarrow X$ be $n_2$. Then consider $C_2$ be the collection of all $s$ sheeted cover of $X$ for $s\le n_2$. Again check if there exists a $k$-orbital cover in $C_2$. If not then $Y_2$ minimal $k$-orbital cover of $X$. Otherwise proceed similarly to more lower sheeted covering. Since there are only finitely many covers of each sheeted so the process will terminate. This proves  Theorem \ref{thm-main4}
\end{proof}
Now to answer of the last part of Question \ref{ques} we prove the following,
\begin{claim}\label{lem1}
Let $X=E_j/K$ be a $m$-orbital map. Then there exists a group $\widetilde{G} \le$ Aut($E_j$) such that $E(E_j)$ has $m$ $\widetilde{G}$-orbits.
\end{claim}
\begin{proof}
Let $X$ be a semi-equivelar toroidal map of type $[p_1^{r_1},p_2^{r_2},\dots p_k^{r_k}]$. By proposition \ref{propo-1} we get $X = E_j/K$ for some discrete subgroup $K$ of Aut($E_j$) where $E_j$ is semi-equivelar tilling of $\mathbb{R}^2$ with same vertex type. 
Let $O_1,O_2, \dots O_m$ be $G$-orbits of $E(X)$. Let $\eta:E_j\to X$ be the covering map. Then $\{\eta^{-1}(O_i) | i = 1,2, \dots m\}$ be a partition of $E(E_j)$. 
Aut($X$)$=$Nor($K$)/$K$. 
Now consider $\widetilde{G} = {\rm Nor}(K).$ Then
$E(E_j)$ forms $m$ $\widetilde{G}$-orbits. This proves  Claim \ref{lem1}.
\end{proof}
\begin{lemma}\label{orbb}
Let $X$ be a $m$-orbital edge-homogeneous toroidal map and $Y$ be a $k$-orbital cover of $X$. Then $k\le m$.
\end{lemma}
\begin{proof}
Let $k \geq m+1$. Let $O_1, O_2, \dots, O_{m+1}$
be distinct Aut($Y$)-orbits of $E(Y)$. Let $\eta$ be the covering map. Suppose $a_i \in O_i$ for $i = 1,2, \dots m+1$. Then $\eta(a_i) \in E(X) ~\forall i$. Since $E(X)$ has $m$ orbits so by pigeon hole principle there exists $i,j \in \{1,2,3, \dots m+1\}$ such that $\eta(a_i), \eta(a_j)$ are in same Aut($X$) orbits of $E(X)$. Therefore there exists  $\Upsilon \in$ Aut($X$) such that $\Upsilon(\eta(a_i)) = \eta(a_j)$. Let $\widetilde{\Upsilon}\in$ Aut($Y$) be the preimage of $\upsilon$ under the projection $p:{\rm Aut}(Y)\to {\rm Aut}(X)$. If $a_i$ and $a_j$ belongs to same sheet of the covering $Y\longrightarrow X$ then $\widetilde{\Upsilon}(a_i) = a_j$. If $a_i$ and $a_j$ belongs to two different sheet then apply a suitable translation on $a_j$ and get an element $a_j' \in O(a_j)$ such that $a_i$ and $a_j'$ belongs to same sheet. Therefore in both cases $ \exists ~\widetilde{\Upsilon} \in$ Aut($Y$) such that $\widetilde{\Upsilon}(a_i) = a_j$. This is a contradiction to $a_i$ and $a_j$ are in different orbits. This proves Lemma \ref{orbb}.
\end{proof}

\section{Acknowledgements}

Authors are supported by NBHM, DAE (No. 02011/9/2021-NBHM(R.P.)/R$\&$D-II/9101).

{\small

}

\end{document}